\documentclass[10pt]{article}
\usepackage[utf8]{inputenc}
\usepackage[margin = 1.95 cm]{geometry}
\usepackage{amsfonts,amsthm, enumitem,amsmath}

\usepackage{graphicx, caption}
\usepackage{amssymb}
\usepackage{setspace}
\usepackage{soul}
\usepackage{changepage}
\usepackage{comment}
\usepackage{float}
\usepackage[pdftex,colorlinks,citecolor=blue,bookmarks=false]{hyperref}
\usepackage{lscape}
\usepackage{mathtools}
\usepackage{thm-restate}

\usepackage{bbm}
\usepackage{cleveref}
\usepackage{graphicx}
\usepackage{mathrsfs}
\usepackage{algorithm}
\usepackage{algpseudocode}
\theoremstyle{plain}

\newtheorem*{thm*}{Theorem}
\newtheorem{theorem}{Theorem}
\Crefname{theorem}{Theorem}{Theorems}
\numberwithin{theorem}{section}

\newtheorem*{lem*}{Lemma}
\newtheorem{lemma}[theorem]{Lemma}
\Crefname{lemma}{Lemma}{Lemmas}

\newtheorem*{claim*}{Claim}
\newtheorem{claim}{Claim}[theorem]
\crefname{claim}{Claim}{Claims}
\Crefname{claim}{Claim}{Claims}

\newtheorem{prop}[theorem]{Proposition}
\Crefname{prop}{Proposition}{Propositions}

\newtheorem{corollary}[theorem]{Corollary}
\crefname{corollary}{Corollary}{Corollaries}

\newtheorem{conj}[theorem]{Conjecture}
\crefname{conj}{Conjecture}{Conjectures}

\newtheorem*{conj*}{Conjecture}

\Crefname{qn}{Question}{Questions}

\Crefname{obs}{Observation}{Observations}

\Crefname{ex}{Example}{Examples}

\theoremstyle{definition}

\Crefname{prob}{Problem}{Problems}

\newtheorem{defn}[theorem]{Definition}
\Crefname{defn}{Definition}{Definitions}

\newtheorem*{defn*}{Definition}

\theoremstyle{remark}
\newtheorem*{rem}{Remark}

\renewenvironment{proof}[1][]{\begin{trivlist}
\item[\hspace{\labelsep}{\bf\noindent Proof#1.\/}] }{\qed\end{trivlist}}

\newcommand{\remove}[1]{}

\newcommand{\cH}{\mathcal{H}}
\newcommand{\cC}{\mathcal{C}}
\newcommand{\cM}{\mathcal{M}}

\renewcommand{\P}{\mathcal{P}}

\renewcommand{\setminus}{-}
\newcommand{\Ex}{\mathbb{E}}
\newcommand{\bR}{\mathbb{R}}
\renewcommand{\Pr}{\mathbb{P}}

\newcommand{\comp}[1]{\overline{#1}}

\newcommand{\eps}{\varepsilon}

\renewcommand{\P}{\mathbb{P}}

\DeclareMathOperator{\bin}{Bin}
\DeclareMathOperator{\TF}{TF}
\DeclareMathOperator{\TK}{TK}

\usepackage[square,sort,comma,numbers]{natbib}
\theoremstyle{plain}
\makeatletter
\newcommand{\optionaldesc}[2]{%
  \phantomsection
  #1\protected@edef\@currentlabel{#1}\label{#2}%
}
\makeatother

\title{Nearly Hamilton cycles in sublinear expanders, and applications}

\author{Shoham Letzter\thanks{Department of Mathematics, University College London, Gower Street, London, WC1E 6BT, U.K. Research supported by the Royal Society. Email:~\textbf{s.letzter@ucl.ac.uk}.} \and Abhishek Methuku\thanks{Department of Mathematics, University of Illinois at Urbana–Champaign, Urbana, IL, USA. Research supported by the UIUC Campus Research Board Award RB25050. Email:~\textbf{abhishekmethuku@gmail.com}} \and Benny Sudakov\thanks{Department of Mathematics, ETH, Z\"urich, Switzerland. Research supported in part by SNSF grant 200021-228014. Email:~\textbf{benjamin.sudakov@math.ethz.ch}}}
\date{}

\begin{document}

\maketitle
\begin{abstract}

We develop novel methods for constructing nearly Hamilton cycles in sublinear expanders with good regularity properties, as well as new techniques for finding such expanders in general graphs. These methods are of independent interest due to their potential for various applications to embedding problems in sparse graphs. In particular, using these tools, we make substantial progress towards a twenty-year-old conjecture of Verstra\"ete, which asserts that for any given graph $F$, nearly all vertices of every $d$-regular graph $G$ can be covered by vertex-disjoint $F$-subdivisions. This significantly extends previous work on the conjecture by Kelmans, Mubayi and Sudakov, Alon, and K\"uhn and Osthus. Additionally, we present applications of our methods to two other problems.

\end{abstract}

\section{Introduction}

A \emph{Hamilton cycle} in a graph $G$ is a cycle passing through all vertices of $G$. A graph is called \emph{Hamiltonian} if it admits a Hamilton cycle. Hamiltonicity is a central notion in graph theory. Since deciding whether a given graph contains a Hamilton cycle is known to be NP-complete, much effort has been devoted to obtaining sufficient conditions for the existence of a Hamilton cycle, for example see~\cite{ajtai1985first, chvatal1972note, csaba2016proof, cuckler2009hamiltonian, draganic2024hamiltonicity} and the surveys~\cite{gould2014recent, kuhn2014hamilton}. Most existing Hamiltonicity conditions, such as Dirac's theorem~\cite{dirac1952some}, are typically applicable only to very dense graphs. Therefore, identifying Hamiltonicity conditions that can also be applied to sparse graphs is a topic of significant interest. Over the last 50 years, a major focus of research has been understanding Hamiltonicity in sparse random graphs. Erd\H{o}s and R\'enyi~\cite{Erdos:1960} posed the foundational question of determining the threshold probability for Hamiltonicity in random graphs. After a series of efforts by various researchers, including Korshunov~\cite{MR434878} and P\'osa~\cite{MR389666}, the problem was ultimately resolved by Koml\'os and Szemer\'edi~\cite{komlos1983limit}, and independently by Bollob\'as~\cite{MR777163}. 

\subsection{Long cycles in Expanders}

    Since Hamiltonicity in random graphs is well understood, a key area of research is to look for Hamilton cycles in deterministic graphs that satisfy `pseudorandom' conditions which enable them to mimic the properties of random graphs. This line of inquiry is particularly valuable for various applications such as Hamiltonicity in random Cayley graphs and Alon and Bourgain’s work on additive patterns in multiplicative subgroups~\cite{alon2014additive}. A well-known class of pseudorandom graphs, introduced by Alon, is defined using spectral properties as follows. A graph $G$ is an \emph{$(n, d, \lambda)$-graph} if it is $d$-regular with $n$ vertices and the second largest eigenvalue of $G$ in absolute value is at most $\lambda$. In 2003, Krivelevich and Sudakov, in their influential paper~\cite{krivelevich2003sparse}, proved that if $d$ is sufficiently larger than $\lambda$, then every $(n, d, \lambda)$-graph is Hamiltonian.  In the same paper, they conjectured that there exists $C > 0$ such that if $\frac{d}{\lambda} \ge C$, then every $(n, d, \lambda)$-graph is Hamiltonian. Shortly after this conjecture was stated, several papers, such as \cite{brandt2006global}, considered an even stronger conjecture, singling out the key properties of $(n, d, \lambda)$-graphs that were believed to be instrumental in demonstrating their Hamiltonicity.  This stronger conjecture asserts that there exists a constant $C > 0$ such that every `$C$-expander' is Hamiltonian. A graph $G$ is called a $C$-expander if it satisfies the following two properties: for every subset $X \subseteq V(G)$ with $|X| < \frac{n}{2C}$, the neighbourhood of $X$ satisfies $|N(X)| \geq C|X|$, and there is an edge between any two disjoint sets of at least $\frac{n}{2C}$ vertices. Despite significant attention (see, e.g., \cite{brandt2006global,hefetz2009hamilton, krivelevich2006pseudo, krivelevich2011number, GlockMCS}) and many motivating applications, these two conjectures were only recently resolved by Dragani{\'c}, Montgomery, Munh\'a Correia,  Pokrovskiy and Sudakov~\cite{draganic2024hamiltonicity}.

	Graph expansion is a fundamental concept in graph theory and computer science, with a wide range of applications; see, for example, the comprehensive survey by Hoory, Linial, and Wigderson~\cite{hoory2006expander}. Most of the expanders studied in the literature are constant expanders, defined by their linear expansion (such as the $C$-expanders discussed earlier). More formally, such graphs $G$ satisfy the property $|N_G(U)| \geq \alpha |U|$ for any subset $U \subseteq V(G)$ that is not too large and not too small, where the expansion factor $\alpha$ is some strictly positive absolute constant independent of $G$.  Sublinear expansion is a weaker notion of this classical expansion, introduced by Koml\'os and Szemer\'edi in the `90s~\cite{komlos1994topological, komlos1996topological}, and characterised by taking a much smaller value of $\alpha$. More precisely, if $G$ is a sublinear expander of order $n$, then $\alpha$ can be taken to satisfy $\alpha = \Omega(\frac{1}{(\log n)^2})$. Although sublinear expanders exhibit weaker expansion properties, their key advantage is that they can be found in essentially any graph. This notion has played a central role in the recent resolution of several long-standing conjectures (see, e.g.\ \cite{alon2023essentially, liu2017proof,liu2023solution,bucic2022erdos,montgomery2023proof, chakraborti2024edge} for notable examples and the survey~\cite{letzter2024sublinear} for a rather comprehensive list).

The study of cycles in expanders is a key area of research; see, for example, \cite{friedman2021cycle, MR3967294, haslegrave2022crux}. Notably, a classic result by Krivelevich \cite{krivelevich2019long} establishes that every $n$-vertex expander with expansion factor $\alpha$ contains a cycle of length $\Omega(\alpha n)$. Hence, every $n$-vertex sublinear expander with expansion factor, say $\alpha = \frac{1}{(\log n)^c}$ for some constant $c>0$, contains a cycle of length $\Omega(\frac{n}{(\log n)^c})$. In general, we cannot necessarily guarantee a linear-sized cycle in sublinear expanders, as shown by the imbalanced complete bipartite graph $K_{n, \alpha n}$. In this paper, we prove that, somewhat surprisingly, sublinear expanders with reasonably good regularity properties admit a nearly Hamilton cycle; see \Cref{nearly Ham informal} for an informal statement and \Cref{lem:almostspannningtopological} for the precise formulation. We also show how to find such expanders in general graphs; see \Cref{refininglemma} for an informal statement and \Cref{cor:expandercoverstrong} for the precise formulation.

Using these techniques, we make significant progress towards resolving a long-standing conjecture of Verstra\"ete from 2002 on packing subdivisions in regular graphs, which we address in the next subsection. Our second application concerns the well-known conjecture of Magnant and Martin~\cite{magnant2009note} from 2009, which asserts that any $d$-regular graph can be partitioned into $n/(d+1)$ paths. Recently, Montgomery, M\"uyesser, Pokrovskiy, and Sudakov~\cite{montgomery2024approximate} asymptotically confirmed this conjecture by showing that nearly all vertices of a $d$-regular graph can be partitioned into $\frac{n}{d+1}$ paths. As a simple consequence of our methods, we show that nearly all vertices of a $d$-regular graph with sufficiently large degree can actually be partitioned into $\frac{n}{d+1}$ \emph{cycles} (see the discussion following \Cref{conj:packingsubdivisons} for further details on this conjecture). Finally, our methods can also be used to find a cycle with many chords, addressing a question of Chen, Erd\H{o}s, and Staton~\cite{chen1996proof} from 1996, and recovering—up to a slightly weaker polylogarithmic factor—a recent result by Dragani\'c, Methuku, Munh\'a Correia, and Sudakov~\cite{draganic2023cycles} (see Section~\ref{sec:application:chords} for details on these applications). Given the prominence of sublinear expanders, these new tools are likely to find further applications in future research. 

\subsection{Packing subgraphs in regular graphs} 

Packings in graphs have been extensively studied. Given two graphs $H$ and $G$, an $H$-packing in $G$ is a collection of vertex-disjoint copies of $H$ in $G$. An $H$-packing in $G$ is called \emph{perfect} if it covers all of the vertices of $G$. The celebrated Hajnal--Szemer\'edi theorem~\cite{hajnal1970proof} from 1970 states that every graph whose order $n$ is divisible by $t$ and whose minimum degree is at least $(1 - \frac{1}{t})n$ contains a perfect $K_t$-packing. (The case $k = 3$ was proved earlier by Corr\'adi and Hajnal~\cite{corradi1963maximal}.)

This theorem is best possible in the sense that the bound on the minimum degree cannot be lowered. 
For non-complete graphs $H$, a series of papers—including ones by 
Alon and Yuster~\cite{alon1996h}, 
Koml\'os, S\'ark\H{o}zy, and Szemer\'edi~\cite{komlos2001proof}, 
and Koml\'os~\cite{komlos2000tiling}—
determined the minimum-degree thresholds that force a perfect $H$-packing in a graph, 
culminating in the work of K\"uhn and Osthus~\cite{kuhn2009minimum}, who essentially settled the problem by giving the best possible such condition (up to an additive constant) for any graph $H$, in terms of the so-called \emph{critical chromatic number}.

In view of the above results, rather surprisingly, K\"uhn and Osthus~\cite{kuhn2005packings} showed that if we restrict our attention to packings in regular graphs, then \emph{any} linear bound on the minimum degree guarantees an almost perfect $H$-packing. More precisely, they showed that for every bipartite graph $H$ and every  $0 < c, \alpha \le 1$, every $cn$-regular graph $G$ of sufficiently large order $n$ has an $H$-packing which covers all but at most $\alpha n$ vertices of $G$. Resolving a conjecture of K\"uhn and Osthus~\cite{kuhn2005packings}, in an upcoming paper \cite{LetzterMetukuSudakovdense} the authors show that the bound $\alpha n$ on the number of uncovered vertices can actually be significantly lowered to obtain an $H$-packing which covers all but a constant number of vertices of $G$, which is clearly best possible, in the sense that there is not always an $H$-packing covering all vertices of $G$, and in fact, there are examples where the number of uncovered vertices grows with $|V(H)|$ and $\frac{1}{c}$.

The notion of subdivisions has played an important role in topological graph theory since the seminal result of Kuratowski~\cite{kuratowski1930probleme} from 1930 showing that a graph is planar if and only if it does not contain a $K_5$-subdivision or a $K_{3,3}$-subdivision. Here, for a given graph $F$, an \emph{$F$-subdivision} or a \emph{subdivision of $F$}, denoted by $\TF$ (a topological copy of $F$), is a graph obtained by replacing each edge $uv$ in $F$ with a path with ends $u$ and $v$, such that the internal vertex sets of these paths are pairwise vertex-disjoint and vertex-disjoint from the original vertices of $F$.
These original vertices of $F$ are called the \emph{branch vertices} of $\TF$. One of the most classical results in this area is due to Mader~\cite{MR0220616}, who showed that
there is some $d = d(t)$ such that every graph with an average degree at least $d$ contains a
subdivision of the complete graph $K_t$. Mader~\cite{MR0220616}, and independently Erd\H{o}s and Hajnal~\cite{MR0173247} conjectured that $d(t) = O(t^2)$. In the `90s, Koml\'os and Szemer\'edi~\cite{komlos1994topological, komlos1996topological} (using sublinear expanders), and independently, Bollob\'as and Thomason~\cite{bollobas1998proof} (using different methods) confirmed this conjecture. Since then, various extensions and strengthenings of this result have been studied. For instance, an old conjecture of Thomassen asks for finding a \emph{balanced} subdivision of $K_t$ and a conjecture of Verstra\"ete asks for finding vertex-disjoint isomorphic subdivisions of $K_t$. Recently, these two conjectures have been resolved in \cite{fernandez2023disjoint, liu2017proof}.

In 2002, Verstra\"ete~\cite{verstraete2002note} made the bold conjecture that every $d$-regular graph contains an almost perfect packing of subdivisions. More precisely, given graphs $F$ and $G$, a $\TF$-packing in $G$ is a collection of pairwise vertex-disjoint copies of subdivisions of $F$ in $G$ (which are not required to be isomorphic). Verstra\"ete~\cite{verstraete2002note} observed that in any $d$-regular graph $G$, one can find a $\TF$-packing which covers about half of the vertices of $G$ (by repeatedly removing subdivisions of $F$ from $G$ until we obtain a graph containing no subdivisions of $F$) and made the following conjecture.

\begin{conj}[Verstra\"ete~\cite{verstraete2002note}, 2002]
\label{conj:packingsubdivisons}
For any graph $F$ and any $\eta > 0$, there exists an integer $d_0 = d_0(F, \eta)$ such that, for all $d \ge d_0$, every $d$-regular graph $G$ of order $n$ contains a $\TF$-packing that covers all but at most $\eta n$ vertices of $G$.
\end{conj}

Note that when $F$ is a complete graph of order two or three, 
Conjecture~\ref{conj:packingsubdivisons} reduces to the problem of covering the vertices 
of a regular graph with vertex-disjoint paths or cycles, respectively. 
In this sense, the conjecture can be viewed as a far-reaching approximate extension 
of Petersen's $2$-factor theorem (see~\cite{lovasz1979combinatorial}), 
which states that for $k\ge 1$, every $2k$-regular graph contains a $2$-factor. Problems involving covering the edges or vertices of a graph with paths/cycles have been extensively studied. Perhaps the most famous open problem in this area is the linear arboricity conjecture of Akiyama, Exoo, and Harary~\cite{akiyama1980covering} from 1980, which states that every graph with maximum degree $\Delta$ can be decomposed into at most $\lceil(\Delta + 1)/2\rceil$ path forests. 
This is related to another well-known conjecture, posed by Magnant and Martin~\cite{magnant2009note} in 2009, which states that the vertices of any $d$-regular graph of order $n$ can be covered by at most $n/(d + 1)$ vertex-disjoint paths. Indeed, the linear arboricity conjecture implies Magnant and Martin's conjecture for odd $d$ by noting that the largest path forest in the conjectured decomposition yields the desired collection of paths.
This latter conjecture is still open; see \cite{feige2022path,gruslys2021cycle,montgomery2024approximate} for some interesting recent progress towards it.
Even the much weaker conjecture by Feige and Fuchs~\cite{feige2022path} that every $d$-regular graph of order $n$ can be covered by at most $O(n/(d + 1))$ vertex-disjoint paths (which follows from the linear arboricity conjecture for all $d$) remains wide open.

Conjecture~\ref{conj:packingsubdivisons} was motivated by an old result of J{\o}rgensen and Pyber \cite{jorgensen1990covering} on covering the \emph{edges} of a graph with subdivisions which actually implies that Conjecture~\ref{conj:packingsubdivisons} holds if we do not require the subdivisions of $F$ to be vertex-disjoint (as observed by K\"uhn and Osthus in~\cite{kuhn2005packings}). In the last twenty years, there have been many results showing that Conjecture~\ref{conj:packingsubdivisons} holds in several natural special cases. A result of Kelmans, Mubayi and Sudakov~\cite{kelmans2001asymptotically} shows that Conjecture~\ref{conj:packingsubdivisons} holds when $F$ is a tree. In 2003, Alon~\cite{alon2003problems} proved that Conjecture~\ref{conj:packingsubdivisons} holds when $F$ is a cycle, using careful estimates on permanents. Alon~\cite{alon2003problems} also remarked that this result can be extended to the case when $F$ is a unicyclic graph but that it does not extend to the case of more complicated graphs $F$. 

In 2005, K\"uhn and Osthus~\cite{kuhn2005packings} proved that Conjecture~\ref{conj:packingsubdivisons} holds when $G$ is dense (i.e.,  $d = \Omega(n)$).
However, a significant obstacle in the way of proving the conjecture in full generality is that this proof relies on Szemer\'edi’s regularity lemma and the Blow-up lemma which only apply to dense graphs $G$. In this paper, we substantially improve on their results by showing that Conjecture~\ref{conj:packingsubdivisons} holds in the following stronger form for all graphs $G$ with at least polylogarithmic average degree.

\begin{theorem}
\label{thm:packingsubdivisons}
For any graph $F$ and large enough $n$, every $d$-regular graph $G$ of order $n$ with $d \ge (\log n)^{130}$ contains a $\TF$-packing that covers all but at most $\frac{n}{(\log \log n)^{1/30}}$ vertices of $G$.
\end{theorem}

As mentioned earlier, our proof of \Cref{thm:packingsubdivisons} involves novel methods for finding nearly Hamilton cycles in sublinear expanders with good regularity and techniques for finding such expanders in general graphs. We give a detailed outline of our methods in Section~\ref{sec:proofsketch}. In particular, in \Cref{cor:expandercoverstrong}, we show that almost all vertices of every $d$-regular $n$-vertex graph with $d \ge 2 \log n$ can be covered by nearly-regular sublinear expanders (see \Cref{subsec:expanders} for a formal definition of these expanders). A key feature of this lemma is that it allows us to control the regularity properties of the expanders we obtain. As a consequence, \Cref{cor:findexpander} shows that every $n$-vertex graph with average degree at least $\Omega(d \log n)$ contains a sublinear expander with maximum degree at most $d$ and average degree extremely close to $d$. 
This result and our techniques for finding nearly Hamilton cycles in sublinear expanders have significant potential for further applications (see Section~\ref{sec:application:chords} for some examples).

In the dense case, where $d = \Omega(n)$, K\"uhn and Osthus~\cite{kuhn2005packings} showed that one can even find a perfect $\TK_t$-packing in $G$ when $t = 4$ and $t = 5$. It follows from the work of Gruslys and Letzter \cite{gruslys2021cycle} towards the aforementioned conjecture of Magnant and Martin \cite{magnant2009note} that this holds for $t = 2$ and $t = 3$. K\"uhn and Osthus posed the question of whether this result holds for $t \ge 6$.  Recently, the authors~\cite{LetzterMetukuSudakovdense} answered this question positively using techniques very different from those used in this paper.  It is known that for all $t \ge 3$, we need $d \ge \sqrt{n/2}$ to have a perfect $\TK_t$-packing in $G$; see~\cite{kuhn2005packings}. It would be interesting to determine the exact degree threshold at which a perfect $\TK_t$-packing can be guaranteed in regular graphs.

\subsection{Organization of the paper} 

The rest of the paper is organized as follows. In Section~\ref{sec:proofsketch} we give a detailed sketch of our proof of Theorem~\ref{thm:packingsubdivisons}. In Section~\ref{sec:prelim}, we introduce the notation and the two notions of expansion used throughout the paper, along with the required probabilistic tools. In Section~\ref{sec:packingwithexpanders}, we develop methods for finding expanders with good regularity properties and prove our first key lemma, \Cref{cor:expandercoverstrong}, which shows that one can cover nearly all vertices of a regular graph with such expanders. In Section~\ref{sec:connectingthroughrandom}, we prove \Cref{lem:connecting} which shows that a collection of vertex-disjoint pairs of vertices (satisfying a certain expansion property) can be joined using vertex-disjoint paths through a random vertex subset of a sublinear expander. In Section~\ref{sec:almostspanningFsubdivisioninexpander}, we develop methods for constructing nearly Hamilton cycles in sublinear expanders and use them to prove our second key lemma, \Cref{lem:almostspannningtopological}, which shows that any sufficiently regular sublinear expander contains an almost-spanning $F$-subdivision. In Section~\ref{sec:proofofmainresult}, we put everything together to prove \Cref{thm:packingsubdivisons}. In Sections~\ref{subsec:cyclepartitions} and \ref{subsec:cyclewithmanychords}, we present two additional applications of our methods: one addressing the conjecture of Magnant and Martin~\cite{magnant2009note}, and the other concerning the existence of a cycle with many chords~\cite{chen1996proof}.

\section{Proof sketch}
\label{sec:proofsketch}
In this section, we sketch the main ideas in our proof of Theorem~\ref{thm:packingsubdivisons}. Let $F$ be a given graph and let $G$ be a $d$-regular graph of sufficiently large order $n$.  
Our strategy involves covering nearly all vertices of
$G$ with vertex-disjoint sublinear expanders that are close to regular and finding an almost-spanning subdivision of
$F$ within each such expander. Specifically, we proceed as follows.
\begin{enumerate}
	\item[\textbf{Step 1.}] \label{packingexpanders} Find a collection $\mathcal{H}$ of vertex-disjoint sublinear expanders with good regularity properties (that is, with average degree very close to the maximum degree), covering nearly all vertices of $G$.

	 \item [\textbf{Step 2.}] Show that each expander $H \in \cH$ contains a nearly Hamilton path $P_H$.
	 \item [\textbf{Step 3.}]\label{almostspanningsubdivision} Find a subdivision of $F$ in each expander $H \in \cH$ containing the nearly Hamilton path $P_H$. 
\end{enumerate}

It is easy to see that the $F$-subdivisions given by \textbf{Step 3} together cover nearly all vertices of $G$. A key contribution of our paper is the development of novel techniques for finding nearly Hamilton paths and cycles in sublinear expanders with good regularity properties. Such expanders are also useful for various other applications. The meta-problem of finding sublinear expanders with good regularity properties was therefore raised in \cite{chakraborti2024edge} by Chakraborti, Janzer, Methuku and Montgomery. Our work here also makes progress towards this problem.

A natural strategy for constructing a nearly Hamilton path in a sublinear expander $H \in \cH$ is to start with a small collection of vertex-disjoint paths $P_1, \ldots, P_r$ in $H$, that together cover almost all vertices in $H$, and connect these paths through a small random set $V_0 \in V(H)$ of reserved vertices. Here we need $V_0$ to be small enough because the vertices of $V_0$ that are not used in the connecting paths are left uncovered by the nearly Hamilton path that we aim to find.

More precisely, we start by partitioning $V(H)$ into random sets $V_0, X_1, \ldots, X_t$ such that $V_0 = o(|V(H)|)$ and the sets $X_1, \ldots, X_t$ have roughly the same size. Then, we obtain the paths $P_1, \ldots, P_r$, by finding a largest matching $M_i$ between each pair of sets $X_i, X_{i+1}$, $1 \le i \le t-1$, taking the union of these matchings, and letting $P_1, \ldots, P_r$ be the connected components in the union that intersect all sets $X_i$ (so that each of the paths $P_1, \ldots, P_r$ contains exactly one vertex from each $X_i$).
In order for the matchings $M_i$ to be large enough so that the paths $P_1, \ldots, P_r$ cover nearly all of the vertices of $H$, it is crucial that $H$ has very good regularity properties. In particular, we need the following property for all the sublinear expanders $H \in \cH$. 

\begin{enumerate}[label = $\mathbf{(P)}$]
    \item \label{Hregularity} 
		For every $H \in \cH$ with average degree at least $d (1 - \eps)$ and maximum degree $d$, we have $\eps \ll \frac{1}{t}$.
\end{enumerate}

Indeed, using Vizing's theorem and standard concentration inequalities, it is easy to show that with high probability $|M_i| \ge |X_i|(1-2\eps)$ for each $1 \le i \le t-1$. This implies that the paths $P_1, \ldots, P_r$ cover all but at most $2\eps |V(H) \setminus V_0|$ vertices from each $X_i$, so that up to $2\eps t$ proportion of the vertices in $V(H) \setminus V_0$ may be left uncovered by the paths $P_1, \ldots, P_r$. To make this a small enough proportion of the vertices of $H$, we need $\eps \ll \frac{1}{t}$ as stated in \ref{Hregularity}. 

For connecting the paths $P_1, \ldots, P_r$ using paths through the random set $V_0$, we use some ideas from recent work of Buci\'c and Montgomery~\cite{bucic2022erdos} and Tomon~\cite{tomon2024robust}, showing that random vertex subsets in sublinear expanders with polylogarithmic average degree are likely to inherit some expansion properties. These ideas were slightly refined in~\cite{chakraborti2024edge} (see, e.g., \cite[Lemma 8]{chakraborti2024edge}) to show that a collection of vertex-disjoint pairs of vertices can be connected through a random subset of vertices of a sublinear expander (using vertex-disjoint paths) provided that the size of the random subset is sufficiently large compared to the number of pairs of vertices. Crucially, this means that for connecting the paths $P_1, \ldots, P_r$ through the random set $V_0$, we need the sets $X_1, \ldots, X_t$ to be much smaller than $V_0$. Since $|X_i| \approx |V(H)|/t$ and $V_0$ must be small, this implies that $1/t$ must be very small, which, in turn, requires $\eps$ to be small enough for \ref{Hregularity} to hold. This explains why we need the sublinear expanders $H \in \cH$ to have extremely good regularity properties. 

In any $d$-regular graph $G$ on $n$ vertices, it is easy to find a sublinear expander with an average degree at least $d(1 - 2\lambda \log n)$ and an expansion factor of $\lambda = O\left(\frac{1}{\log n}\right)$, using standard methods such as iteratively removing sparse cuts. In fact, one can cover nearly all vertices of our $d$-regular graph $G$ with such sublinear expanders. By choosing $\lambda$ sufficiently small, these expanders indeed exhibit the desired strong regularity properties. However, these expanders may contain only a small number of vertices of $G$, so the expansion characterized by this $\lambda$ might be too weak for following the aforementioned strategy to complete \textbf{Step 2}. To address this, we introduce a \emph{refining procedure}, detailed at the end of this section, that begins with these expanders and gradually enhances their expansion while essentially preserving their strong regularity properties. Roughly speaking, this enables us to prove the following lemma; for a precise statement, see \Cref{cor:expandercoverstrong}.

\begin{lemma}
\label{refininglemma}
For any $c_1 \ge 2$, there is a constant $c_2 > 0$ such that the following holds. Nearly all vertices of every $d$-regular graph $G$ with sufficiently large degree can be covered by vertex-disjoint (robust) sublinear expanders $H$ with an average degree of at least $d(1 - \frac{1}{(\log |V(H)|)^{c_1}})$ and an expansion factor of $\frac{1}{(\log |V(H)|)^{c_2}}$.
\end{lemma}

This lemma shows that by taking $c_1$ sufficiently large, we can obtain sublinear expanders $H$ with sufficiently strong regularity (although at the expense of slightly weaker expansion). Unfortunately, this improved regularity of our expanders $H$ is still insufficient to satisfy \ref{Hregularity} because the methods of~\cite{bucic2022erdos, chakraborti2024edge} require each of the sets $X_1, \ldots, X_t$ to be significantly smaller than $V_0$ (depending on the parameter $c_1$) -- see, e.g., the proof of \cite[Lemma 8]{chakraborti2024edge}. 
Therefore, \Cref{refininglemma} still does not allow us to directly connect the paths $P_1, \ldots, P_r$ through the random set $V_0$, making the construction of a nearly Hamilton path challenging with the strategy described above. To overcome this difficulty, our main idea is to iteratively connect the paths $P_1, \ldots, P_r$ through $V_0$ using a procedure that, in each iteration, either directly connects a good proportion of the paths or identifies a well-expanding subset of their leaves (see Figure~\ref{fig:cherryvspaths}). More precisely, in each iteration of the procedure, we first greedily connect the paths through $V_0$
using as many vertex-disjoint paths of length two as possible (a similar idea for connecting paths, in a different context, was recently used in~\cite{montgomery2024approximate}).  Crucially, when it is no longer possible to connect using vertex-disjoint paths of length two, we can identify a small subset $S$ of the leaves of the paths $P_1, \ldots, P_r$ that expands very well into $V_0$. The improved regularity of expanders $H$ (provided by \Cref{refininglemma}) together with a variant of a lemma from~\cite{letzter2024separating} is now sufficient to connect the paths whose leaves lie in the subset $S$ (through $V_0$). This allows us to join a good proportion of the paths $P_1, \ldots, P_r$ through $V_0$ in each iteration. By iterating this process $\Theta(\log |V(H)|)$ times, we eventually obtain a nearly Hamilton path $P_H$ in $H$, completing~\textbf{Step~2}. 

\begin{figure}[h]
    \centering
	\includegraphics[width=0.85 \textwidth]{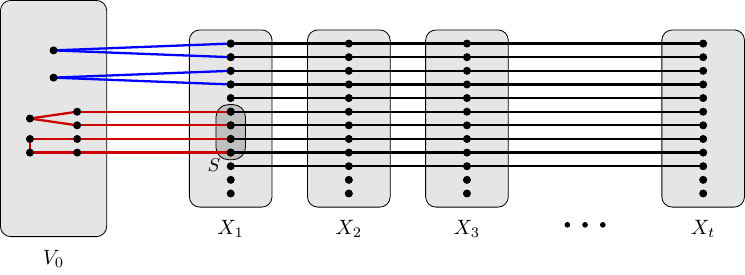}
    \caption{If we cannot find sufficiently many connecting paths of length two (shown in blue), we identify a set $S$ of leaves that expands very well into $V_0$ and connect vertices in $S$ using paths (shown in red) through $V_0$.}
\label{fig:cherryvspaths}
\end{figure}

Since the procedure requires $\Theta(\log |V(H)|)$ iterations to connect the paths $P_1, \ldots, P_r$ through $V_0$, the sets $X_1, \ldots, X_t$ must be at least $\Theta(\log |V(H)|)$ times smaller than $V_0$. This constraint demands an even smaller choice of $\eps$ to satisfy \ref{Hregularity}. Nevertheless, by selecting $c_1$ sufficiently large in \Cref{refininglemma}, we can still apply the iterative procedure to connect the paths and construct a nearly Hamilton path, even when the expanders $H$ possess slightly weaker expansion properties. This yields the following lemma (see \Cref{lem:almostspannningtopological} for the precise formulation).

\begin{lemma}
\label{nearly Ham informal}
Let $c > 0$ be fixed, let $0 < \eps < \frac{1}{(\log n)^4}$, let $n$ be sufficiently large, and let $d \geq (\log n)^{10c+51}$. Then, every $n$-vertex (robust) sublinear expander $H$ with an expansion factor of $\frac{1}{(\log n)^c}$, average degree at least $d(1-\eps)$ and maximum degree at most $d$, contains a cycle (and thus a path $P_H$) of length at least $n - \frac{n}{\log n}$.  
\end{lemma}

Next, we \emph{absorb} the vertices of the path $P_H$ into a subdivision of $F$ that is found within a small random subset $R$ of vertices in each of our expanders $H$ (as illustrated in Figure~\ref{fig:almostspanningsubdivision}). (Here, we use the aforementioned techniques for connecting vertex pairs, together with a classical result for finding subdivisions~\cite{bollobas1998proof, komlos1996topological}.) This yields the desired almost-spanning subdivision of $F$ in each of the expanders $H \in \cH$, thus completing \textbf{Step 3} and the proof of Theorem~\ref{thm:packingsubdivisons}.

As discussed earlier, sublinear expanders with strong regularity properties hold potential for a wide range of applications (some of which are discussed in \Cref{sec:application:chords}). Given this independent interest, we conclude this section with a sketch of the proof of \Cref{refininglemma}, showing that almost all vertices of a nearly regular graph can be covered with vertex-disjoint nearly regular expanders.
We note that this lemma, together with known methods for `regularising' a graph, imply the existence of a nearly regular expander in every graph with sufficiently large average degree (see \Cref{cor:findexpander}), a potentially very useful tool for applications.

\vspace{-4 mm}
\paragraph{A refining procedure.}  Let $G$ be a graph with an average degree at least $d(1 - \eps)$ and a maximum degree $d$. By repeatedly removing sparse cuts, it is easy to find an expander in $G$ with an expansion factor of $\lambda = O(\frac{1}{\log n})$ and an average degree of at least $d(1 - 2 \lambda \log n)$. By iteratively finding such expanders, removing their vertices and continuing with the remaining graph, we can cover nearly all vertices
of the graph $G$ using a collection $\cH_1$ of vertex-disjoint expanders with an average degree of at least $d(1 - \eps \log n)$ and an expansion factor of $\frac{\eps}{\log n}$ (see \Cref{lem:expandercoverweak}). 
However, since these expanders may contain very few vertices compared to $n = |V(G)|$, the expansion factor $\frac{\eps}{\log n}$ may represent only very weak expansion. Such a weak expansion is insufficient for most applications; in particular, we require a much stronger expansion to construct a nearly Hamilton cycle within these expanders. Let $C$ be a large constant. If $\eps = \frac{1}{(\log n)^{C-1}}$, then the expanders in $\cH_1$ have an average degree of at least $d\big(1 - \frac{1}{(\log n)^{C-2}}\big)$ and an expansion factor of $\frac{1}{(\log n)^C}$. This means that, although the expanders in $\cH_1$ may exhibit very weak expansion, they possess excellent regularity properties. We introduce a refining procedure that begins with the expanders in $\cH_1$ and iteratively improves their expansion while largely preserving their strong regularity properties, ultimately producing expanders with both good expansion and regularity properties.

The main idea of this refining procedure is as follows. We carefully choose certain `thresholds' $n_t \leq \dots \leq n_1 = n$ for vertex sizes, where $\log n_{i+1} = (\log n_i)^{\frac{C-2}{C-1}}$, and regularity thresholds $\eps_1 \leq \dots \leq \eps_t$, where $\eps_i = \frac{1}{(\log n_i)^{C-2}} = \frac{1}{(\log n_{i+1})^{C-1}}$. We iteratively construct a sequence of collections $\mathcal{H}_i$, $1 \leq i \leq t$, of expanders where each collection covers nearly all vertices of $G$, such that for each $i$, the expanders in $\mathcal{H}_{i+1}$ have slightly better expansion and only slightly weaker regularity properties compared to the expanders in the collection $\mathcal{H}_i$. We achieve this by `refining' any expander $H \in \mathcal{H}_i$ that has too few vertices (using
the aforementioned fact that any graph with average degree at least $d(1 - \eps)$ and maximum degree $d$ can be covered with vertex-disjoint expanders having an average degree of at least $d(1 - \eps \log n)$ and an expansion factor of $\frac{\eps}{\log n}$). More precisely, we replace any expander $H \in \mathcal{H}_i$ that has fewer than $n_{i+1}$ vertices with a new collection $\mathcal{H}_{i+1}(H)$ of vertex-disjoint expanders that have improved expansion (and only slightly weaker regularity properties) covering nearly all vertices of $H$, and
we let $\mathcal{H}_{i+1} = \bigcup_{H \in \mathcal{H}_i} \mathcal{H}_{i+1}(H)$ be the resulting collection of expanders. 
Crucially, the parameters are set up so that, if $H \in \cH_i$ and $|V(H)| < n_{i+1}$, then the expanders in $\cH_{i+1}(H)$ have an \emph{improved} expansion factor of $\frac{1}{(\log n_{i+1})^C}$ while maintaining an average degree at least $d(1 - \eps_{i+1})$. 
By repeating this refining procedure we eventually obtain a collection $\cH_t = \cH$ of expanders that cannot be refined further. This means that for every expander $H \in \cH$, there is a step $j$ at which it was last refined, where $n_j > |V(H)| \ge n_{j+1}$ and $H \in \cH_j \setminus \cH_{j-1}$. It is then easy to see that $H$ has an expansion factor of $\frac{1}{(\log n_{j})^C} = \frac{1}{(\log n_{j+1})^{C(C-1)/(C-2)}} \ge \frac{1}{(\log |V(H)|)^{C(C-1)/(C-2)}}$, and an average degree of at least $d(1 - \eps_j) = d(1 - \frac{1}{(\log n_j)^{C-2}}) \ge d(1 - \frac{1}{(\log |V(H)|)^{C-2}})$, as desired, proving \Cref{refininglemma}.    

\section{Preliminaries}
\label{sec:prelim}

\subsection{Notation} We write $c = a \pm b$ if $a-b \le c \le a+b$.
For a set $U \subseteq V(G)$, let $\comp{U}$ denote the set $V(G) \setminus U$. For a set $S \subseteq V(G)$, let $G \setminus S$ denote the subgraph of $G$ induced by $V(G) \setminus S$.
By $e(G)$, we denote the number of edges of $G$, and for $S\subseteq V(G)$, we denote by $e_G(S)$ the number of edges of $G$ induced by $S$. For disjoint sets $A,B\subseteq V(G)$, let $e_G(A,B)$ denote the number of edges of $G$ 
with one endpoint in $A$ and the other in $B$, and let $G[A,B]$ denote the bipartite 
subgraph of $G$ with parts $A$ and $B$.

For a graph $G$, we denote by $d(G)$ its average degree, by $\delta(G)$ its minimum degree and by $\Delta(G)$ its maximum degree. For a vertex $v \in V (G)$, we denote its degree by $d_G(v)$ and the set of its neighbours by
$N_G(v)$. For a set of vertices $X \subseteq V(G)$, and an integer $i \ge 0$, let $N_G(X)$ denote the set of vertices outside of $X$ that are adjacent to at least one vertex of $X$, and we write $B^i_G(X)$ to be the set of vertices at distance at most $i$ from $X$ in $G$. We often omit subscripts and write, e.g. $N(X)$ instead of $N_G(X)$, if it is clear from the context which graph we are working with. Given vertices $x, y$, an $(x,y)$-path is a path from $x$ to $y$.

All logarithms in this paper are base $2$. When dealing with large numbers, we often omit floor and ceiling signs whenever they are not crucial.

\subsection{Expansion} \label{subsec:expanders}
In this paper, we use two different notions of expansion that are closely related to each other. Here, we introduce these two notions and then prove a lemma that relates them.
We will use the following standard notion of edge expansion.

\begin{defn}[$\lambda$-expander]
	Let $\lambda > 0$. We say that a graph $H$ is a $\lambda$-expander if every set $U \subseteq V(H)$ with $|U| \le \frac{1}{2}|V(H)|$ satisfies $e(U, \overline{U}) \ge \lambda d(H) |U|$.  
\end{defn}

We will mostly use the following notion of vertex expansion, which is a generalization of the notion of expansion introduced by Buci\'c and Montgomery~\cite{bucic2022erdos}. Related notions of expansion were recently developed by Shapira and Sudakov~\cite{shapira2015small}, by Haslegrave, Kim, and Liu~\cite{haslegrave2022extremal}, and by Sudakov and Tomon~\cite{sudakov2022extremal}. 

\begin{defn}[$(\eps, c, s)$-expander]
\label{defn:robust-sublinear-expansion} 
	An graph $H$ is called an $(\eps, c, s)$-expander if, for every $U\subseteq V(H)$ and $F\subseteq E(H)$ with $1\le |U|\leq \frac23 |V(H)|$ and $|F|\leq s|U|$, we have
	\begin{equation}
		|N_{H-F}(U)|\geq \frac{\eps}{(\log |V(H)|)^c} \cdot |U|.\label{eqn:expands}
	\end{equation}
\end{defn}

Expanders as in \Cref{defn:robust-sublinear-expansion} are sometimes called \emph{robust} sublinear expanders, referring to the graph $F$ (Koml\'os and Szemer\'edi's definition of sublinear expanders did not include this feature). Notice that the expander becomes more `robust' the larger $s$ is.

The main difference between the notion of expansion we use (as given in \Cref{defn:robust-sublinear-expansion}) and  the one used in \cite{bucic2022erdos} is that we use an additional parameter $c$ for our expanders (which is set to $2$ in \cite{bucic2022erdos}). This parameter measures the rate of expansion, where larger values of 
$c$ correspond to weaker expansion. Our approach in this paper relies crucially on expanders with excellent regularity properties, specifically those with average degree very close to maximum degree. In \Cref{lem:expandercoverstrong}, we construct such expanders, which allows us to enhance their regularity by increasing the parameter $c$ which, however, leads to a slight weakening of the expansion. 

The following lemma is standard and connects the two notions of expansion for (almost) regular expanders.

\begin{lemma}
	\label{lem:relateexpanders}
	Let $0 < \eps \le 1/4$, $c > 0$, let $n$ be large and let $\lambda = \frac{1}{(\log n)^c}$. Let $H$ be an $n$-vertex $\lambda$-expander with $\Delta(H) \le d$ and $d(H) \ge d(1 - \eps)$. Then, $H$ is a $(\frac{1}{8}, c, \frac{\lambda d}{4})$-expander. 
\end{lemma}
\begin{proof}
	Let $U \subseteq V(H)$ with $1 \le |U| \le \frac{2}{3} n$, and let $F \subseteq E(H)$ with $|F| \le \frac{\lambda d}{4} |U|$. 

	First consider the case when $1 \le |U| \le n/2$. Then, since $H$ is a $\lambda$-expander, we have $e_H(U, \overline{U}) \ge \lambda d(H) |U|$. Therefore, $e_{H-F}(U, \overline{U}) \ge (\lambda d(H) - \frac{\lambda d}{4}) |U| \ge (\lambda d(1 - \eps) - \frac{\lambda d}{4}) |U| \ge \frac{\lambda d}{4} |U|$. Since $\Delta(H) \le d$, this implies that 
	\begin{equation*}
		|N_{H - F}(U)| \ge \frac{\lambda}{4} |U| \ge \frac{|U|}{8 (\log n)^c}.
	\end{equation*}
	Now consider the case when $n/2 \le |U| \le 2n/3$. So $|U|/2 \le n/3 \le |\overline{U}| \le n/2$. Then, $e_{H-F}(U, \overline{U}) \ge \lambda d(H) |\overline{U}| - \frac{\lambda d}{4} |U| \ge \frac{\lambda d(H)}{2} |U| - \frac{\lambda d}{4} |U| \ge (\frac{\lambda d(1 - \eps) }{2}- \frac{\lambda d}{4}) |U| \ge \frac{\lambda d}{8} |U|.$ Since $\Delta(H) \le d$, we have 
	$$ |N_{H - F}(U)| \ge \frac{\lambda}{8} |U| = \frac{|U|}{8 (\log n)^c}.$$
	Therefore, $H$ is a $(\frac{1}{8}, c, \frac{\lambda d}{4})$-expander, proving the lemma.
\end{proof}

\subsection{Probabilistic tools}\label{sec:conc}
We will often use a basic version of Chernoff's inequality for the binomial random variable (see, for example, \cite{alon2016probabilistic, upfal2005probability}).

\begin{theorem}[Chernoff's bound]\label{chernoff}
Let $n$ be an integer, let $0\le p \le 1$, let $X \sim \bin(n,p),$ and let $\mu=\mathbb{E} X = np$. Then, the following hold.

\begin{enumerate}[label = \rm(\roman*)]
    \item If $0 < \delta < 1$, then $$\P(X \le (1-\delta) \mu) \le e^{-\frac{\delta^2\mu}{2}}.$$
    \item If $\delta \ge 0$, then $$\P(X\geq (1+\delta )\mu )\le e^{-\frac{\delta^{2}\mu}{(2+\delta )}}.$$
\end{enumerate}
\end{theorem}

Additionally, we use the following martingale concentration result (see Chapter 7 in \cite{alon2016probabilistic}).
We say that a function $f : \prod_{i = 1}^n	\Omega_i \to \bR$, where $\Omega_i$ are arbitrary sets, is \emph{$k$-Lipschitz} if $|f(u) - f(v)| \le k$ for every $u, v \in \prod_{i = 1}^n \Omega_i$ that differ in at most one coordinate.
\begin{lemma} 
	\label{lem:concentration}
	Let $X_1, \ldots, X_n$ be independent random variables, with $X_i$ taking values in a set $\Omega_i$ for $i \in [n]$, and write $X = (X_1, \ldots, X_n)$. 
	Suppose that $f : \prod_{i = 1}^n \Omega_i \to \bR$ is $k$-Lipschitz. Then,
	\begin{equation*}
		\Pr\big(|f(X) - \Ex f(X)| > t\big) \le 2 \exp \left(\frac{-t^2}{2k^2n}\right).
	\end{equation*}
\end{lemma}

\subsection{A regularisation theorem} \label{subsec:regularisation}

Our main theorem (\Cref{thm:packingsubdivisons}) is stated for regular graphs $G$ with sufficiently large degree. However, for other applications of our methods (such as the one that will be discussed in \Cref{sec:application:chords}), the graph
$G$ does not need to be regular. In such cases, it is useful to extract a regular subgraph whose degree remains close to the average degree of $G$. We achieve this using a result by Chakraborti, Janzer, Methuku, and Montgomery \cite{CJMMregular}. 

\begin{theorem} \label{thm:regular}
	There exists a constant $\gamma$ such for all positive integers $r,n$ with $r \le n/2$, every $n$-vertex graph with average degree at least $\gamma r\log(n/r)$ contains an $r$-regular subgraph.
\end{theorem}

In \cite{CJMMregular}, \Cref{thm:regular} and a related result for the case when $r$ is small relative to $n$ were used to resolve a problem of R\"odl and Wysocka~\cite{rodl1997note} from 1997 for almost all $r$ and to obtain tight bounds for the Erd\H{o}s-Sauer problem~\cite{erdos1975some} (up to an absolute constant factor), improving the recent breakthrough of Janzer and Sudakov~\cite{janzer2023resolution}, who had resolved the problem up to a constant depending on $r$. For our application in this paper, it suffices to find an almost-regular subgraph with a large degree rather than a fully regular subgraph. This can be achieved, at the cost of an additional polylogarithmic factor, using a simple lemma by Buci\'c, Kwan, Pokrovskiy, Sudakov, Tran, and Wagner~\cite{bucic2020nearly}, which builds on methods developed by Pyber~\cite{pyber1985regular}.

\section{Packing a regular graph with nearly regular expanders}
\label{sec:packingwithexpanders}

In this section, we prove our first key lemma, \Cref{cor:expandercoverstrong}, which shows that one can find sufficiently regular, vertex-disjoint expanders covering almost all vertices of any regular graph with sufficiently large degree. An important feature of this lemma is that it allows us to control the regularity properties of the expanders we obtain (via the parameter $C$), which is crucial to constructing a nearly Hamilton cycle in each of these expanders (as explained in the proof sketch, see \Cref{sec:proofsketch}).
As mentioned earlier, expanders with good regularity properties are useful for various other applications, and the question of finding such expanders was raised by Chakraborti, Janzer, Methuku, and Montgomery in \cite{chakraborti2024edge}. Our lemma below makes progress on this problem.

\begin{lemma}
	\label{cor:expandercoverstrong}
	Let $\alpha > 0$ be a fixed real number, let $\eps, C, n, d$ satisfy $d \ge 2 \log n$, $0 \le \eps \le (\log n)^{-(C-1)}$, $C \ge \max\{28\alpha + 3, 56\alpha + 1\}$ and let $n$ be large enough. Let $c = \frac{C(C-1)}{C-28 \alpha - 1}$. 

	Suppose that $G$ is an $n$-vertex graph with $\Delta(G) \le d$ and $d(G) \ge d(1- \eps)$. Then, there is a collection $\cH$ of vertex-disjoint subgraphs of $G$ such that every $H \in \cH$ is a $(\frac{1}{8}, c, s_H)$-expander satisfying $d(H) \ge d(1 - \eps_H)$ and $\delta(H) \ge d(H)/2$, where $s_H \coloneqq \frac{d}{4 (\log |V(H)|)^c}$, and $\eps_H \coloneqq (\log |V(H)|)^{-(C-28 \alpha - 1)}$. Moreover, $\sum_{H \in \cH} |V(H)| \ge (1 - \frac{(\log \log \log n)^2}{(\log \log n)^{\alpha}})n$. 
\end{lemma}

Note that increasing the parameter $C$ in \Cref{cor:expandercoverstrong} improves the regularity properties of the expanders we obtain at the cost of slightly weakening their expansion. Indeed, in our proof we let $\alpha  = \frac{1}{28}$ (though this choice is somewhat arbitrary), so that $c = \frac{C(C-1)}{C- 2}$. Hence, as $C$ increases, it is easy to see that $\eps_H = (\log |V(H)|)^{-(C-2)}$ decreases (improving the regularity properties of our expanders) and $c$ increases (making their expansion weaker).

In certain cases, we do not require the full strength of \Cref{cor:expandercoverstrong} and instead seek a single expander that is nearly regular. In such situations, we can actually remove the near-regularity assumption on $G$ at the cost of a slightly worse lower bound on the average degree of $G$, as shown below.

\begin{corollary}
	\label{cor:findexpander}
	There is a constant $\gamma$ such that the following holds for all sufficiently large $n$ and for $C, d$ satisfying $C \ge 4$, $d \ge 2 \log n$.
	Suppose that $G$ is an $n$-vertex graph with $d(G) \ge \gamma d \log n$. Then, there is a subgraph $H \subseteq G$ which is a $(\frac{1}{8}, c, s)$-expander satisfying $\Delta(H) \le d$, $d(H) \ge d(1 - \mu)$ and $\delta(H) \ge d(H)/2$, where $c \coloneqq \frac{C(C-1)}{C-2}$, $s \coloneqq \frac{d}{4 (\log |V(H)|)^c}$, and $\mu \coloneqq (\log |V(H)|)^{-(C-2)}$. 
\end{corollary}

\begin{proof}[ of \Cref{cor:findexpander} using \Cref{cor:expandercoverstrong}]
	Apply the regularisation theorem, \Cref{thm:regular}, to $G$ to obtain a $d$-regular subgraph $G'$. Now apply \Cref{cor:expandercoverstrong} to the graph $G'$ (with $\alpha = \frac{1}{28}$ and $\eps = 0$), and let $H$ be any graph in the collection $\cH$ guaranteed by the lemma. Then, it is easy to see that $H$ has the desired properties.
\end{proof}

\begin{rem}
	Notice that in the proof of \Cref{cor:findexpander} we used \Cref{thm:regular}, due to Chakraborti, Janzer, Methuku and Montgomery \cite{CJMMregular}, which gives a tight bound on the number of edges in an $n$-vertex graph needed to guarantee the existence of a $d$-regular subgraph. A slightly weaker version of \Cref{cor:findexpander} can be obtained using more elementary tools, where $d(G)$ is required to be at least $5d(\log n)^{C}$ rather than $\gamma d\log n$, by applying \cite[Lemma~2.2]{bucic2020nearly}, as mentioned earlier.
\end{rem}

We prove \Cref{cor:expandercoverstrong} by combining Lemma~\ref{lem:relateexpanders} with the following variant of \Cref{cor:expandercoverstrong} for edge expanders rather than (robust) vertex expanders.

\begin{lemma}
	\label{lem:expandercoverstrong}
	Let  $\alpha > 0$ be a fixed real number, let $\eps, C, n, d$ satisfy $d \ge 2\log n$, $0 \le \eps \le (\log n)^{-(C-1)}$, $C \ge \max\{28\alpha + 3, 56\alpha + 1\}$, and let $n$ be large enough. Let $c = \frac{C(C-1)}{C-28\alpha-1}$.

	Let $G$ be an $n$-vertex graph with $\Delta(G) \le d$ and $d(G) \ge d(1- \eps)$. Then, there is a collection $\cH$ of vertex-disjoint subgraphs of $G$ such that every $H \in \cH$ is a $\lambda_H$-expander satisfying $d(H) \ge d(1 - \eps_H)$ and $\delta(H) \ge d(H)/2$, where $\lambda_H \coloneqq (\log |V(H)|)^{-c}$ and $\eps_H \coloneqq (\log |V(H)|)^{-(C-28 \alpha - 1)}$. Moreover, $\sum_{H \in \cH} |V(H)| \ge (1 - \frac{(\log \log \log n)^2}{(\log \log n)^{\alpha}})n$. 
\end{lemma}

The rest of this section is dedicated to proving Lemma~\ref{lem:expandercoverstrong}. 
To that end, in \Cref{subsec:weakexpanderpacking}, we first construct an almost-perfect packing using expanders that may exhibit very weak expansion but possess excellent regularity properties. Then, in \Cref{subsec:refiningthepacking}, we introduce a refining procedure that begins with these expanders and iteratively improves their expansion while largely preserving their strong regularity properties, ultimately producing expanders with both good expansion and regularity properties, thereby proving Lemma~\ref{lem:expandercoverstrong}.

\subsection{Packing with very weak expanders}
\label{subsec:weakexpanderpacking}

In this subsection, we prove the following lemma which shows that one can cover almost all vertices of an $n$-vertex graph $G$ (with average degree at least $d(1 - \eps)$ and maximum degree $d$) using
vertex-disjoint $\lambda$-expanders having expansion factor $\lambda = O(\frac{\eps}{\log n})$ and having average degree very close to that of $G$. Since these expanders could have very few vertices compared to $n$, this $\lambda$ may indicate very weak expansion (which is not sufficient for our purpose), but crucially, these expanders have good regularity properties if we pick $\eps$ small enough. This is crucial to our refining procedure in \Cref{subsec:refiningthepacking}, which starts with these expanders and iteratively improves their expansion while only slightly weakening their regularity properties.

\begin{lemma}
	\label{lem:expandercoverweak}
	Let $\alpha > 0$ be a fixed real number and let $\eps, \lambda, n$ satisfy $\lambda \log n \le \min \{\frac{1}{10}, \eps \}$ and let $n$ be sufficiently large. Let $G$ be an $n$-vertex graph with $\Delta(G) \le d$ and $d(G) \ge d(1-\eps)$. Then, there exists a collection $\cH$ of vertex-disjoint $\lambda$-expanders in $G$ such that $d(H) \ge d(1 - \eps (\log n)^{28 \alpha})$ and $\delta(H) \ge d(H)/2$ for every $H \in \cH$ and $\sum_{H \in \cH} |V(H)| \ge (1 - \frac{1}{(\log n)^{\alpha}}) n$. 
\end{lemma}

We build up to the proof of Lemma~\ref{lem:expandercoverweak} using a sequence of lemmas as follows.
\begin{itemize}
    \item First, in \Cref{lem::findexpander}, we show how to find one $\lambda$-expander in a sufficiently regular graph $G$ using standard methods (see, e.g., \cite{shapira2015small, draganic2023cycles}).
    \item Then, by iteratively applying Lemma~\ref{lem::findexpander}, we show how to cover a good proportion of vertices of $G$ with vertex-disjoint $\lambda$-expanders in \Cref{lem:packingoneround}. 
    \item  Finally, by iteratively applying \Cref{lem:packingoneround} we show how to cover almost all vertices of $G$ with vertex-disjoint $\lambda$-expanders (with average degree very close to that of $G$), proving Lemma~\ref{lem:expandercoverweak}.
\end{itemize}

\begin{lemma}
	\label{lem::findexpander}
	Let $G$ be an $n$-vertex graph with $d(G) = d$, and let $\lambda \log n \le \frac{1}{10}$. Then, $G$ contains a $\lambda$-expander $H$ with $d(H) \ge d(1 - 2 \lambda \log n)$ and $\delta(H) \ge d(H)/2$.
\end{lemma}
\begin{proof}
	We perform a procedure which finds the desired $\lambda$-expander $H$ in $G$. Before we can describe the procedure, we need the following claim.
	\begin{claim}
		\label{claim:iterateexpander}
		Let $F$ be a subgraph of $G$ that is not a $\lambda$-expander. 
Then, there exists a non-empty set $U \subseteq V(F)$ with 
$|U| \le \frac{|V(F)|}{2}$ such that either 
$d(F[\overline U]) \ge d(F)$ or 
$d(F[U]) \ge (1-2\lambda)d(F)$, 
where $\overline U \coloneqq V(F)\setminus U$.
	\end{claim}
	\begin{proof}[ of claim] 
		Since $F$ is not a $\lambda$-expander, there is a non-empty set $U \subseteq V(F)$ with $|U| \leq \frac{|V(F)|}{2}$ such that $e_F(U,\overline{U}) < \lambda d(F) |U|$. We will show that $U$ is the desired subset. Suppose for a contradiction that $d(F[\overline{U}]) < d(F)$ and $d(F[U])< (1-2\lambda)d(F)$. Then, using the bound on $e_F(U,\overline{U})$, we have $e(F) = e_F(U) + e_F(U,\overline{U}) + e_F(\overline{U})  < |U| d(F) \left( \frac{1-2\lambda}{2} + \lambda \right) + |\overline{U}| \frac{d(F)}{2} = |V(F)| \frac{d(F)}{2} = e(F)$, which is a contradiction.   
	\end{proof}

	Let us now describe the procedure. We start the procedure with $F = G$. At every step of the procedure, we consider a subgraph $F$ of $G$ and do the following. 

	\begin{itemize}
		\item If $F$ is a $\lambda$-expander with $\delta(F) \ge d(F)/2$, then we let $H \coloneqq F$ be the desired $\lambda$-expander and stop the procedure. 
		\item If $F$ has a vertex $v$ with degree less than $d(F)/2$, we remove it and define $F' \coloneqq F \setminus v$. Note that in this case $d(F') \ge d(F)$. Now we repeat the procedure with $F'$ playing the role of $F$.
  
		\item Otherwise, by Claim~\ref{claim:iterateexpander}, there is a non-empty set $U \subseteq V(F)$ with $|U| \leq \frac{|V(F)|}{2}$ such that either $d(F[\overline{U}]) \geq d(F)$ or $d(F[U]) \geq (1-2\lambda)d(F)$. In the former case, let $F' \coloneqq F[\overline{U}]$, and in the latter case, let $F' \coloneqq F[U]$. Then, we repeat the procedure with $F'$ playing the role of $F$.
	\end{itemize}

	Note that at any step of our procedure, we can have $d(F') < d(F)$ only if $|V(F')| \le |V(F)|/2$ and $d(F') \geq (1-2\lambda)d(F)$. Furthermore, since $|V(F')| \le |V(F)|/2$ can only occur for at most $\log n$ steps, at any step of the procedure the subgraph $F$ we consider satisfies $d(F) \geq (1-2\lambda)^{\log n} d(G) \ge (1 - 2 \lambda \log n)d(G) \ge 4d(G)/5$. Combining this with the fact that the number of vertices of $F$ strictly decreases after each step, it follows that the procedure eventually stops with a non-empty subgraph $H$ which is a $\lambda$-expander satisfying $d(H) \ge (1 - 2 \lambda \log n)d(G)$ and $\delta(H) \ge d(H)/2$. This proves the lemma.
\end{proof}

By repeatedly applying \Cref{lem::findexpander}, we obtain the following lemma which shows that one can cover a good proportion of the vertices of any (sufficiently regular) graph $G$ with vertex-disjoint $\lambda$-expanders (whose average degree is very close to that of $G$). 

\begin{lemma}
	\label{lem:packingoneround}
	Let $\eps, \lambda, n$ satisfy $\lambda \log n \le \min \{\frac{1}{10}, \eps \}$. Let $G$ be an $n$-vertex graph with $\Delta(G) \le d$ and $d(G) \ge d(1- \eps)$. Then, there exists a collection $\cH$ of vertex-disjoint $\lambda$-expanders in $G$ such that $d(H) \ge d(1-10 \eps)$ and $\delta(H) \ge d(H)/2$ for every $H \in \cH$ and $\sum_{H \in \cH} |V(H)| \ge n/4$.
\end{lemma}
\begin{proof}
	Let $\cH$ be a maximal collection of vertex-disjoint $\lambda$-expanders in $G$ with average degree at least $d(1-10 \eps)$. Let $U \coloneqq \bigcup_{H \in \cH} V(H)$ and let $\overline{U} \coloneqq V(G) \setminus U$. If $|U| \ge n/4$, then the $\lambda$-expanders in $\cH$ satisfy the desired properties, proving the lemma. 

	So suppose $|U| < n/4$. Then, $e(U) + e(U, \overline{U}) \le d|U| - e(U) \le d|U| - \frac{d(1-10\eps)}{2} |U| = \frac{d(1+10\eps)}{2} |U|$. Hence, 
	\begin{align*}
		e(\overline{U}) 
		& \ge e(G) - \left(e(U) + e(U, \overline{U})\right) 
		\ge \frac{d(1-\eps)}{2}n - \frac{d(1+10\eps)}{2} |U| \\
		& = \frac{d}{2} \big((1-\eps)n - (1+10\eps)|U|\big) 
		= \frac{d}{2} \left(|\overline{U}| - \eps n - 10 \eps |U|\right) \\
		& \ge \frac{d}{2} \left(|\overline{U}| - \eps n - 10 \eps \frac{n}{4}\right) 
		\ge \frac{d}{2} \left(|\overline{U}| - 8 \eps |\overline{U}|\right) 
		= \frac{d}{2}(1-8 \eps)|\overline{U}|.
	\end{align*}
	Therefore, $d(\overline{U}) \ge d (1 - 8 \eps).$ Hence, by applying Lemma~\ref{lem::findexpander} to $G[\overline{U}]$, we obtain a $\lambda$-expander $H$ with $d(H) \ge d(\overline{U})(1 - 2 \lambda \log n) \ge d (1 - 8 \eps) (1 - 2 \lambda \log n) \ge d (1 - 8 \eps -2 \lambda \log n) \ge d (1 - 10 \eps)$ and $\delta(H) \ge d(H)/2$. This contradicts the maximality of the collection $\cH$ of vertex-disjoint $\lambda$-expanders. Therefore, $|U| \ge n/4,$ proving the lemma.
\end{proof}


Finally, we prove \Cref{lem:expandercoverweak} by repeatedly applying \Cref{lem:packingoneround}.

\begin{proof}[ of \Cref{lem:expandercoverweak}]
	Let $C = 100$. For every $i \ge 1$, let $\eps_i = C^{i} \eps$, and let $\cH_i$ be a collection of vertex-disjoint $\lambda$-expanders with average degree at least $d(1- \eps_i)$ in $G \setminus (\bigcup_{H \in \cH_1 \cup \ldots \cup \cH_{i-1}} V(H))$ that maximises the number of vertices covered by expanders in $\cH_i$ among all such collections. For all $i \ge 1$, let $V_i \coloneqq \bigcup_{H \in \cH_i} V(H)$. 

	\begin{claim}
		\label{claim:coverbyexpanders}
		For all $i \ge 1$, $\sum_{k = 1}^i |V_k| \ge (1 - (\frac{3}{4})^{i})n$.
	\end{claim}
	\begin{proof}[ of claim] 
		We prove the claim by induction. For $i = 1$, the claim immediately follows by applying Lemma~\ref{lem:packingoneround}
		to $G$ to obtain a collection $\cH_1$ of vertex-disjoint $\lambda$-expanders covering at least $n/4$ vertices and satisfying $d(H) \ge d(1 - 10 \eps) \ge d(1 - C \eps) = d (1 - \eps_1)$ and $\delta(H) \ge d(H)/2$ for every $H \in \cH_1$. This shows that $|V_1| \ge n/4$, as desired.

		Now suppose that for all $1 \le j \le i$, we have $\sum_{k = 1}^j |V_k| \ge (1 - (\frac{3}{4})^j)n$. We will show that $\sum_{k = 1}^{i+1} |V_k| \ge (1 - (\frac{3}{4})^{i+1})n$. 

		Let $U \coloneqq V_1 \cup \ldots \cup V_i$ and let $\overline{U} \coloneqq V(G) \setminus U$. Since for each $1 \le k \le i$, the average degree of the $\lambda$-expanders in $\cH_k$ (whose union spans the vertex set $V_k$) is at least $d(1- \eps_k)$, we have $e(U) \ge \sum_{k=1}^i \frac{d}{2} (1- \eps_k) |V_k| = \frac{d}{2} |U| - \frac{d}{2} \sum_{k=1}^i \eps_k|V_k|$. Hence, 
		\begin{equation}
			\label{eq:claim2:1}
			e(U) + e(U, \overline{U}) \le d |U| - e(U) \le \frac{d}{2} |U| + \frac{d}{2} \sum_{k=1}^i \eps_k|V_k|.
		\end{equation}

Note that since $\eps_1 \le \ldots \le \eps_i$ and $\sum_{k = 1}^j |V_k| \ge (1 - \left(\frac{3}{4}\right)^j)n$ for $1 \le j \le i$, we can bound the sum $\sum_{k=1}^i \eps_k|V_k|$ using the proposition below. 

\begin{prop}\label{prop:weighted-sum-bound}
Let $\eps_1\le \cdots \le \eps_i$ and let $a_1,\ldots,a_i\ge 0$ satisfy $\sum_{k=1}^i a_k \le n$ and
$\sum_{k=1}^j a_k \ge \big(1-(\tfrac34)^j\big)n$ for all $1\le j\le i$.
Then,
\[
\sum_{k=1}^i \eps_k a_k
\le
\sum_{k=1}^{i-1}
\eps_k(\tfrac34)^{k-1}\tfrac n4
+\eps_i(\tfrac34)^{i-1}n .
\]
\end{prop}

\begin{proof}
Note that since $\eps_1\le\cdots\le\eps_i$, the sum 
$\sum_{k=1}^i \eps_k a_k$ 
is maximised when, for each $1\le k\le i$, the value of $a_k$ is taken to be as large 
as permitted by the inequalities 
$\sum_{j=1}^k a_j \ge (1-(\tfrac34)^k)n$ and $\sum_{k=1}^ia_k \le n$, 
assuming that the values of $a_{k+1},\ldots,a_i$ have already been fixed at their 
largest admissible values.

From $\sum_{k=1}^{i-1} a_k \ge (1-(\tfrac34)^{i-1})n$ and $\sum_{k=1}^i a_k \le n$, we obtain
$a_i \le (\tfrac34)^{i-1}n$. By choosing $a_i$ as large as possible, we obtain $a_i = (\tfrac34)^{i-1}n$, which implies $\sum_{k=1}^{i-1}a_k = (1 - (\tfrac34)^{i-1})n$.
Similarly, using
$\sum_{k=1}^{i-2} a_k \ge (1-(\tfrac34)^{i-2})n$ we get
\[
a_{i-1}
\le
\big((\tfrac34)^{i-2}-(\tfrac34)^{i-1}\big)n
=
(\tfrac34)^{i-2}\tfrac n4 .
\]

Again, by choosing $a_{i-1}$ as large as possible, we obtain $a_{i-1} =(\tfrac34)^{i-2}\tfrac n4$. Proceeding inductively, for every $1\le k\le i-1$, we obtain
$a_k = (\tfrac34)^{k-1}\tfrac n4$.
Substituting these values of $a_k$ for $1 \le k \le i-1$, and $a_i$ yields
\[
\sum_{k=1}^i \eps_k a_k
\le
\sum_{k=1}^{i-1}
\eps_k(\tfrac34)^{k-1}\tfrac n4
+\eps_i(\tfrac34)^{i-1}n ,
\]
as required.
\end{proof}
By Proposition~\ref{prop:weighted-sum-bound} (applied with $a_k=|V_k|$), we have
		\begin{align}
			\label{eq:claim2:2}
			\begin{split}
				\sum_{k=1}^i \eps_k|V_k| 
				& \le \sum_{k=1}^{i-1} \eps_k \left(\frac{3}{4}\right)^{k-1} \frac{n}{4} + \eps_i \left(\frac{3}{4}\right)^{i-1} n  \\
				& \le \sum_{k=1}^{i} \frac{4}{3} \eps \left(\frac{3 C}{4}\right)^{k} n 
				\le \frac{\frac{4}{3} \eps}{\frac{3 C}{4}-1} \left(\frac{3 C}{4}\right)^{i+1}n.
			\end{split}
		\end{align}
		Combining \eqref{eq:claim2:1} and \eqref{eq:claim2:2}, we have $$e(U) + e(U, \overline{U}) \le \frac{d}{2} |U| + \frac{\frac{2}{3} \eps d}{\frac{3 C}{4}-1} \left(\frac{3 C}{4}\right)^{i+1}n.$$
		Therefore, 
		\begin{align} \label{claim2:eq3}
			\begin{split}
				e(\overline{U}) \ge e(G) - (e(U) + e(U, \overline{U})) 
				&\ge \frac{1}{2} nd(1-\eps) - \frac{d}{2} |U| - \frac{\frac{2}{3} \eps d}{\frac{3 C}{4}-1} \left(\frac{3 C}{4}\right)^{i+1}n \\  
				&\ge \frac{d}{2} |\overline{U}| - \frac{1}{2} n d \eps - \frac{8\eps d}{9C - 12} \left(\frac{3 C}{4}\right)^{i+1}n. 
			\end{split}
		\end{align}

		Since $\sum_{k = 1}^i |V_k| \ge (1 - (\frac{3}{4})^i)n$, we have $|\overline{U}| \le (\frac{3}{4})^i n$. Thus, by \eqref{claim2:eq3}, 
		\begin{align*}
			d(\overline{U}) 
			= \frac{2 e(\overline{U})}{|\overline{U}|} 
			& \ge d\left(1 - \left(\frac{4}{3}\right)^i \eps \left(1 + \frac{16}{9C - 12} \left(\frac{3 C}{4}\right)^{i+1}  \right) \right) \\
			& = d \left(1 - \eps \left(\left(\frac{4}{3}\right)^i + \frac{4 C^{i+1}}{3C - 4}\right) \right).
		\end{align*}

		Note that since $C = 100$, we have $\left(\frac{4}{3}\right)^i + \frac{4 C^{i+1}}{3C - 4} \le \frac{C^{i+1}}{10}$, so
		\begin{equation*}
			d(\overline{U}) \ge d \left(1 - \frac{\eps C^{i+1}}{10}\right)  = d \left(1 - \frac{\eps_{i+1}}{10}\right).
		\end{equation*}
		Now, note that $\lambda \log |\overline{U}| \le \lambda \log n \le \min \{ \frac{1}{10},  \eps \} \le \min \{ \frac{1}{10},  \frac{C^{i+1} \eps}{10} \} = \min \{ \frac{1}{10},  \frac{\eps_{i+1}}{10} \}$, where in the last inequality we used $C = 100$. Hence, by Lemma~\ref{lem:packingoneround}, there exists a collection $\cH_{i+1}$ of vertex-disjoint $\lambda$-expanders in $G[\overline{U}]$ with $d(H) \ge d(1 - \eps_{i+1})$, and $\delta(H) \ge d(H)/2$ for every $H \in \cH_{i+1}$ such that, if $V_{i+1} = \bigcup_{H \in \cH_{i+1}} V(H)$, then $|V_{i+1}| = |\bigcup_{H \in \cH_{i+1}} V(H)|\ge |\overline{U}|/4$. Therefore, $n - \sum_{k = 1}^{i+1} |V_k| \le \frac{3}{4}|\overline{U}|$. Since $|\overline{U}| \le (\frac{3}{4})^i n$, we have 
		$ n - \sum_{k = 1}^{i+1} |V_k| \le \frac{3}{4}|\overline{U}| \le (\frac{3}{4})^{i+1}n.$ This shows that $\sum_{k = 1}^{i+1} |V_k| \ge (1-(\frac{3}{4})^{i+1})n$, proving the claim.
	\end{proof}

Let $t = \min \{ i \mid  (\frac{3}{4})^{i+1} \le \frac{1}{(\log n)^{\alpha}}\}$. Then, $(\frac{3}{4})^t \ge \frac{1}{(\log n)^{\alpha}}$, so $(\frac{4}{3})^t \le (\log n)^{\alpha}$ which implies that $t \le \frac{\alpha }{\log (4/3)} \log\log n \le 3 \alpha \log\log n$. Thus, $C^{t+1} \le C^{3 \alpha \log\log n + 1} \le C^{4 \alpha \log\log n} \le 2^{28 \alpha \log\log n}  = (\log n)^{28 \alpha}.$ 

Now consider the collection of $\lambda$-expanders $\cH \coloneqq \cH_1 \cup \ldots \cup \cH_{t+1}$. Every $H \in \cH$ is a $\lambda$-expander with $d(H) \ge d(1 - \eps_{t+1}) = d(1 - \eps C^{t+1}) \ge d(1 - \eps (\log n)^{28 \alpha})$, $\delta(H) \ge d(H)/2$, and by Claim~\ref{claim:coverbyexpanders}, these expanders cover at least $\sum_{k = 1}^{t+1} |V_k| \ge (1 - (\frac{3}{4})^{t+1})n \ge (1 - \frac{1}{(\log n)^{\alpha}})n$ vertices of $G$, proving the lemma.
\end{proof}

\subsection{Refining procedure}
 \label{subsec:refiningthepacking}

To prove Lemma~\ref{lem:expandercoverstrong} we use a refining procedure that gradually improves the expansion of the expanders given by Lemma~\ref{lem:expandercoverweak} without significantly weakening their regularity. This is achieved through repeated applications of Lemma~\ref{lem:expandercoverweak}. The main idea of the refining procedure is described in the proof sketch at the end of Section~\ref{sec:proofsketch}.

\begin{proof}[ of Lemma~\ref{lem:expandercoverstrong}]
	For convenience, let $\beta \coloneqq 28 \alpha$ and $\gamma \coloneqq \frac{C - \beta - 1}{C-1}$. 
	Note that by the assumptions of the lemma, $\gamma \ge 1/2$. Let $n_1 = n$, and for every $i \ge 1$, let $n_{i+1}$ satisfy $\log n_{i+1} = (\log n_{i})^{\gamma}$ (so that $\log n_{i+1} = (\log n)^{\gamma^{i}}$). Let $\eps_0 = (\log n_1)^{-(C - 1)}$, and for every $i \ge 1$, let $\eps_i = (\log n_i)^{-(C-\beta-1)}$ and $\lambda_i = (\log n_i)^{-C}$. 

	Let $\cH_0 \coloneqq \{G\}$. Since $d(G) \ge d(1 - \eps_0)$, and $\lambda_1 \log n_1 = (\log n_1)^{-(C - 1)} = \eps_0 \le \frac{1}{10}$, applying Lemma~\ref{lem:expandercoverweak} with $G, \lambda_{1},$ $\alpha$ and $\eps_{0}$ playing the roles of $G, \lambda,$ $\alpha$ and $\eps$ respectively, we obtain a collection $\cH_1$ of $\lambda_1$-expanders such that, for every $H \in \cH_1$, 
we have $d(H) \ge d(1-\eps_{0}(\log n_1)^{28\alpha}) = d(1-\eps_1)$ and 
$\delta(H) \ge d(H)/2$, and moreover $\sum_{H \in \cH_1} |V(H)| \ge (1-\frac{1}{(\log n)^{\alpha}})n$.

	With $\cH_1$ defined as above, we define a sequence $\cH_2, \cH_3, \ldots$, as follows.  For every $i \ge 1$ such that $\eps_i \le \frac{1}{10}$, assuming that $\cH_i$ was already defined, we do the following for each $H \in \cH_i$.
	\begin{itemize}
		\item[$\mathrm{(a)}_i$] If $|V(H)| \le n_{i+1}$ and $d(H) \ge d(1-\eps_{i})$, then let $\cH_{i+1}(H)$ be a collection of vertex-disjoint $\lambda_{i+1}$-expanders in $H$ satisfying $d(F) \ge d(1-\eps_{i} (\log |V(H)|)^{28 \alpha})$ and $\delta(F) \ge d(F)/2$ for every $F \in \cH_{i+1}(H)$, and moreover, $\sum_{F \in \cH_{i+1}(H)} |V(F)| \ge (1 - \frac{1}{(\log |V(H)|)^{\alpha}}) |V(H)|$ (obtained by applying Lemma~\ref{lem:expandercoverweak} with $H, \lambda_{i+1},$ $\alpha$ and $\eps_{i}$ playing the roles of $G, \lambda,$ $\alpha$ and $\eps$ respectively). Note that Lemma~\ref{lem:expandercoverweak} is indeed applicable because $\lambda_{i+1} \log |V(H)| \le \lambda_{i+1} \log n_{i+1} = (\log n_{i+1})^{-(C-1)} = (\log n_{i})^{-\gamma(C-1)} = (\log n_{i})^{-(C - \beta - 1)} = \eps_i \le \frac{1}{10}$.

		\item[$\mathrm{(b)}_i$] Otherwise, let $\cH_{i+1}(H) = \{H\}$.
	\end{itemize}

	Then, we define $\cH_{i+1} \coloneqq \bigcup_{H \in \cH_i} \cH_{i+1}(H)$. We will now show that the following claim holds.

	\begin{claim}
		\label{claim:refiningcovers}
		For every $i \ge 1$, if $\eps_{i-1} \le \frac{1}{10}$, then the following holds.
		\begin{itemize}
			\item[$\mathrm{(A)}_i$] For every $H \in \cH_{i}$, we have $d(H) \ge d(1 - \eps_{i})$ and $\delta(H) \ge d(H)/2$. 

			\item[$\mathrm{(B)}_i$] $\sum_{H \in \cH_{i}} |V(H)| \ge (1 - \frac{i}{(\log \log n)^{\alpha}})n.$
		\end{itemize}
	\end{claim}
	\begin{proof}[ of claim]
		We prove the claim by induction. By our choice of $\cH_{1}$ as discussed above, we know that $\mathrm{(A)}_1$ holds, and that $\sum_{H \in \cH_1} |V(H)| \ge (1-\frac{1}{(\log n)^{\alpha}})n \ge (1 - \frac{1}{(\log \log n)^{\alpha}})n$, so $\mathrm{(B)}_1$ holds as well. 
		Now suppose $\mathrm{(A)}_i$ and $\mathrm{(B)}_i$ hold for $i = k$ for some $k \ge 1$, and we will show that they hold for $i = k+1$. Suppose that $\eps_k \le \frac{1}{10}$ (otherwise there is nothing to prove).

		Indeed, for every $H \in \cH_k$ with $|V(H)| > n_{k+1}$, it follows from $\mathrm{(b)}_k$ that $\cH_{k+1}(H) = \{H\}$. So we have $d(H) \ge d(1 - \eps_k) \ge d(1 - \eps_{k+1})$ and $\delta(H) \ge d(H)/2$ by $\mathrm{(A)}_k$. For every $H \in \cH_k$ with $|V(H)| \le n_{k+1}$, it follows from $\mathrm{(a)}_k$ and $\mathrm{(A)}_k$ that $d(F) \ge d(1-\eps_{k} (\log |V(H)|)^{28 \alpha})$ and $\delta(F) \ge d(F)/2$ for every $F \in \cH_{k+1}(H)$. Moreover, notice that we have 
		\begin{align*}
			\eps_{k} (\log |V(H)|)^{28 \alpha} 
			& \le \eps_{k} (\log n_{k+1})^{28 \alpha} \\
			& = (\log n_k)^{-(C-\beta-1)} (\log n_{k+1})^{\beta} \\
			& = (\log n_{k+1})^{-\frac{(C-\beta-1)}{\gamma}} (\log n_{k+1})^{\beta} 
			= (\log n_{k+1})^{-(C-\beta -1)} 
			= \eps_{k+1}.
		\end{align*}
		Therefore, for every $F \in \cH_{k+1}(H)$, we have $d(F) \ge d(1 - \eps_{k+1})$ in both cases. This proves that $\mathrm{(A)}_{k+1}$ holds. 
		Now our goal is to prove $\mathrm{(B)}_{k+1}$.  Note that 
		\begin{equation}
			\label{eq:uncoveredverticesHk+1}
			\sum_{H \in \cH_{k+1}} |V(H)| = \sum_{H \in \cH_{k}} \sum_{F \in \cH_{k+1}(H)} |V(F)| \ge  \sum_{H \in \cH_k} \left(1 - \frac{1}{(\log |V(H)|)^{\alpha}} \right) |V(H)|, 
		\end{equation}
		where, in the last inequality we used $\mathrm{(a)}_k$ in the case when $|V(H)| \le n_{k+1}$, and in the case when $|V(H)| > n_{k+1}$, $\mathrm{(b)}_k$ trivially implies that  $\sum_{F \in \cH_{k+1}(H)} |V(F)| = |V(H)|$.  Now note that, for every $H \in \cH_k$ we have $|V(H)| \ge d(H) \ge d(1 - \eps_k) \ge \frac{9 d }{10} \ge \log n$. Hence, $1 - \frac{1}{(\log |V(H)|)^{\alpha}} \ge 1 - \frac{1}{(\log \log n)^{\alpha}}$. Therefore, by \eqref{eq:uncoveredverticesHk+1}, and our assumption that $\mathrm{(B)}_k$ holds, we have 
		$$ \sum_{H \in \cH_{k+1}} |V(H)| \ge \left(1 - \frac{1}{(\log \log n)^{\alpha}}\right) \left(1 - \frac{k}{(\log \log n)^{\alpha}} \right) n \ge \left(1 - \frac{k+1}{(\log \log n)^{\alpha}}\right) n.$$
		This proves that $\mathrm{(B)}_{k+1}$ holds, and completes the proof of the claim.
	\end{proof}
	Let $t = \max \{ i \mid \log n_i \ge 4 \}$. Then, $4 \ge \log n_{t+1} = (\log n_t)^{\gamma} \ge (\log n_t)^{1/2} \ge 2.$ In particular, $n_{t+1} \le 16$. Since, by the assumptions of Lemma~\ref{lem:expandercoverstrong}, $C \ge \beta + 3$, we have $\eps_t = (\log n_t)^{-(C-\beta-1)} \le 4^{-(C-\beta-1)} \le \frac{1}{10}$. Moreover, since $\eps_{t-1} \le \eps_t$, we also have $\eps_{t-1} \le \frac{1}{10}$, so by $\mathrm{(A)}_t$ of Claim~\ref{claim:refiningcovers}, for every $H \in \cH_{t}$, we have $d(H) \ge d(1 - \eps_{t}) \ge d(1 - \frac{1}{10}) > 16 \ge n_{t+1}$ and $\delta(H) \ge d(H)/2$. Therefore, for every $H \in \cH_t$, we have $|V(H)| \ge d(H) > n_{t+1}$ and thus $\cH_{t+1} = \cH_t$. 
	
    Now let $H \in \cH_t$, and let $j$ be the smallest index such that $H \in \cH_j$. 
	Then, $H$ is a $\lambda_j$-expander with $d(H) \ge d(1 - \eps_j)$.
	We claim that $n_j \ge |V(H)| > n_{j+1}$. Indeed, the upper bound clearly holds if $j = 1$. Otherwise, since $H \notin \cH_{j-1}$, there exists $H' \in \cH_{j-1}$ such that $|V(H')| \le n_j$ and $H \in \cH_j(H')$. Since $H$ is a subgraph of $H'$, it follows that $|V(H)| \le n_j$, as desired.
	For the lower bound, notice that $H \in \cH_j \cap \cH_{j+1}$ (otherwise, $H$ would not be in $\cH_{t}$), showing that $|V(H)| > n_{j+1}$.
	Recall that, as defined in the statement of Lemma~\ref{lem:expandercoverstrong}, $\lambda_H = (\log |V(H)|)^{-\frac{C(C-1)}{C-28 \alpha - 1}}$, and $\eps_H = (\log |V(H)|)^{-(C-28 \alpha - 1)}$.
	Hence,
	\begin{align*}
		\lambda_j 
		= (\log n_j)^{-C} 
		= (\log n_{j+1})^{-C/\gamma} 
		= (\log n_{j+1})^{\frac{-C(C-1)}{C-\beta-1}} 
		\ge (\log |V(H)|)^{\frac{-C(C-1)}{C-\beta-1}} 
		= \lambda_H,
	\end{align*}
	and that 
	\begin{equation*}
		\eps_j = (\log n_j)^{-(C-\beta-1)} \le (\log |V(H)|)^{-(C-\beta-1)} = \eps_H.
	\end{equation*}
	Hence, every $H \in \cH_t$ is a $\lambda_H$-expander with $d(H) \ge d(1 - \eps_H)$ and $\delta(H) \ge d(H)/2$. Finally, since $(\log n)^{\gamma^{t-1}} = \log n_{t} \ge 4$, we have $t \le \frac{\log \log \log n}{\log (1/\gamma)}+1 \le (\log \log \log n)^2$. Thus, by $\mathrm{(B)}_t$ of Claim~\ref{claim:refiningcovers}, we have $\sum_{H \in \cH_{t}} |V(H)| \ge (1 - \frac{t}{(\log \log n)^{\alpha}})n \ge (1 - \frac{(\log \log \log n)^2}{(\log \log n)^{\alpha}})n.$ This shows that $\cH_t$ is a collection of vertex-disjoint subgraphs of $G$ satisfying the desired properties, completing the proof of Lemma~\ref{lem:expandercoverstrong}.
\end{proof}

\section{Connecting vertex pairs through a random vertex subset of an expander}
\label{sec:connectingthroughrandom}

In this section, we prove the following generalization of a lemma from \cite{letzter2024separating}, which allows for connecting pairs of vertices $(x_1, y_1), \ldots, (x_r, y_r)$ in a robust sublinear expander through a random set of vertices using vertex-disjoint paths, if all subsets of the set $\{x_1, \ldots, x_r, y_1, \ldots, y_r\}$ expand well. This is a variant of Theorem~16 in \cite{bucic2022erdos}, which only requires the paths to be edge-disjoint. The proof follows \cite{bucic2022erdos} closely, which, in turn, uses some ideas from Tomon~\cite{tomon2024robust}.

	\begin{lemma} \label{lem:connecting}
		Let $2^{-9} \le \eps < 1$, $c > 0$, $0 < q < 1$, $s \ge \frac{2 (\log n)^{9c+21}}{q^{10}}$, and let $n$ be sufficiently large. Suppose that $G$ is an $n$-vertex $(\eps, c, s)$-expander, and let $V$ be a random subset of $V(G)$ obtained by including each vertex independently with probability $q$. Then, with probability at least $1 - \frac{1}{n}$, the following holds. If $x_1, \ldots, x_r, y_1, \ldots, y_r$ are distinct vertices not contained in $V$ such that $|N(X)| \ge \frac{100 (\log n)^{7c+19}}{q^6}|X|$ for every $X \subseteq \{x_1, \ldots, x_r, y_1, \ldots, y_r\}$, then there is a sequence of paths $P_1, \ldots, P_r$, each of length at most $(\log n)^{c+4}$, such that for each $i \in [r]$, $P_i$ is a path from $x_i$ to $y_i$ with internal vertices in $V$, and the paths $P_1, \ldots, P_r$ are pairwise vertex-disjoint.
	\end{lemma}

	We will introduce the relevant preliminaries in \Cref{subsec:prelims}, and then prove \Cref{lem:connecting} in three steps, as detailed in \Cref{subsec:random-step-one,subsec:expansion-step-two,subsec:expansion-step-three}.

\subsection{Preliminaries} \label{subsec:prelims}

We will use \Cref{prop:strong-expansion} and \Cref{prop:stars}, which show that every vertex set $U$ in an expander (provided $U$ is not too large) either \emph{expands well}, meaning it has a large neighbourhood, or \emph{expands robustly}, meaning that there are many vertices in $N(U)$ with many neighbours in $U$. These two statements are slight variations of propositions from~\cite{bucic2022erdos}. The main differences are that here we do not consider a set of forbidden edges (as we do not need this feature) and that we introduce a new parameter $c$ controlling the rate of expansion of our expanders.

First, we adapt Proposition 12 from \cite{bucic2022erdos}. For any graph $G$, parameter $d$, and $U \subseteq V(G)$, let $N_{G,d}(U) \coloneqq \{v \in V (G) \setminus U : |N_G(v) \cap U| \ge d\}$ i.e., $N_{G,d}(U)$ is the set of vertices in $G$ outside $U$ which have degree at least $d$ in $U$.

	\begin{prop}
		\label{prop:strong-expansion}
		Let $\eps, c > 0$, and let $0 < d \le s$. Suppose that $G$ is an $n$-vertex $(\eps, c, s)$-expander. Then, for every set $U \subseteq V(G)$ with $|U| \le \frac{2n}{3}$, either
		\[
			\text{\textbf{\emph{a)}} }\;\;|N_{G}(U)| \ge \frac{s|U|}{2d},\;\;\;\text{ or }\;\;\;\text{\textbf{\emph{b)}}}\;\; |N_{G,d}(U)|\ge \frac{\eps|U|}{(\log n)^{c}}.
		\]
	\end{prop}
	\begin{proof}
		Suppose \textbf{a)} is not satisfied, so that $|N_{G}(U)|<\frac{s|U|}{2d}$. Let $X=N_{G}(U)\setminus N_{G,d}(U)$, so that $|X|<\frac{s|U|}{2d}$. Let $F$ be the edges of $G$ between $U$ and $X$, so that $|F| < |X|d\leq s|U|/2$. Note that, by the definition of $F$, we have $N_{G,d}(U)=N_{G-F}(U)$. As $G$ is an $(\eps, c, s)$-expander, we have
		\[
			|N_{G,d}(U)|=|N_{G-F}(U)|\geq \frac{\eps|U|}{(\log n)^{c}},
		\]
		and therefore \textbf{b)} holds, as required. 
	\end{proof}

	Now we adapt Proposition 13 from \cite{bucic2022erdos} which shows that more structure can be found in both outcomes of the above proposition.

	\begin{lemma} \label{prop:stars}
		There is an $n_0$ such that the following holds for all $n\geq n_0$, $2^{-9} < \eps< 1$, $c > 0$, $r, t \geq (\log n)^2$ and $s\geq 20rt$. Let $G$ be an $n$-vertex $(\eps, c, s)$-expander, let $U \subseteq V (G)$ satisfy $|U| \leq 2n/3$. Then, in $G$ we can find either
		\begin{enumerate}[label=\rm{(\alph*)}]
			\item $\frac{|U|}{10r}$ pairwise vertex-disjoint stars of size $t$, whose centers are in $U$ and whose leaves are in $V(G)\setminus U$, or
			\item \label{propstars:findH} a bipartite subgraph $H$ with vertex classes $U$ and $X\subseteq V(G)\setminus U$ such that
				\begin{itemize}
					\item $|X|\geq \frac{\eps |U|}{2 (\log n)^c}$ and
					\item every vertex in $X$ has degree at least $r$ in $H$ and every vertex in $U$ has degree at most $2t$ in $H$.
				\end{itemize}
		\end{enumerate}
	\end{lemma}

	\begin{proof}
		Take a maximal collection $\mathcal C$ of pairwise vertex-disjoint stars in $G$ with $t$ leaves, centres in $U$ and leaves outside of $U$. Let $C \subseteq U$ be the set of centres of these stars and $L\subseteq V(G)\setminus U$ be the set consisting of all their leaves. Suppose \textbf{a)} does not hold. Then, we can assume that $|C| \le \frac{|U|}{10r}$ and thus $|L| = |C|\cdot t \le \frac{|U|}{10r}\cdot t$, and, by the maximality of $\mathcal C$, that there is no vertex in $U\setminus C$ with at least $t$ neighbours in $G$ in $V(G)\setminus (U\cup L)$. Thus,
		\begin{equation}\label{eqn:NGW}
			|N_{G}(U\setminus C)|\leq |C|+|L|+|U\setminus C|\cdot t\leq  \frac{|U|}{10 r}+|C|\cdot t +|U\setminus C|\cdot t< 2|U|\cdot t.
		\end{equation}

		We now construct a set $X\subseteq V(G)\setminus U$ and a bipartite subgraph $H$ with vertex classes $U$ and $X$ using the following process, starting with $X_0=\emptyset$ and setting $H_0$ to be the graph with vertex set $U\cup X_0$ and no edges. Let $k=|V(G)\setminus U|$ and label the vertices of $V(G)\setminus U$ arbitrarily as $v_1,\ldots,v_k$. For each $i\geq 1$, if possible, pick a star $S_i$ in $G$ with centre $v_i$ and $r$ leaves in $U$ such that the vertices in $U$ in the graph $H_{i-1}\cup S_i$ have degree at most $2t$, and let $H_i=H_{i-1}\cup S_i$ and $X_i=X_{i-1} \cup \{v_i\}$, while otherwise we set $H_i=H_{i-1}$ and $X_{i}=X_{i-1}$. Finally, let $H=H_k$ and $X=X_k=V(H_k)\setminus U$. We will now show that \textbf{b)} holds for this choice of $H$ (with vertex classes $U$ and $X$).

		Firstly, observe that every vertex of $U$ has degree at most $2t$ in $H_i$ for each $i\in [k]$ by construction, and that every vertex $v_i$ in $X$ has degree exactly $r$ in $H$, so the second condition in \textbf{b)} holds. Thus, we only need to show that $|X| \ge \frac{\eps|U|}{2(\log n)^c}$ holds.
  
		To see this, let $U'$ be the set of vertices in $U \setminus C$ with degree exactly $2t$ in $H$. As each vertex in $U\setminus C$ has fewer than $t$ neighbours in $G$ in $X\setminus L$ (due to the maximality of the collection of stars $\mathcal C$), the vertices in $U'$ must have at least $t$ neighbours in $H$ in $X\cap L$. As each vertex in $X\cap L$ has $r$ neighbours in $H$, we have
		\[
			|U'|\leq \frac{r|X \cap L|}{t}\le \frac{r}{t}\cdot |L| \leq \frac{r}{t}\cdot \frac{|U|\cdot t}{10r}= \frac{|U|}{10}.
		\]

		Let $B=C \cup U'$, so that
		\[
			|B|\leq \frac{|U|}{10r}+\frac{|U|}{10}\le \frac{|U|}{2},
		\]
		and, thus, $|U\setminus B|\geq \frac{|U|}2$.

		Then, by \Cref{prop:strong-expansion} applied to $U\setminus B$ with $d=r$, we have either $|N_{G}(U\setminus B)|\geq \frac{s|U\setminus B|}{2r}$ or $|N_{G,r}(U\setminus B)|\geq \frac{\eps |U\setminus B|}{(\log n)^c}$. As $$\frac{s|U\setminus B|}{2r}\geq \frac{s|U|}{4r}\geq 5t|U|,$$ the former inequality contradicts \eqref{eqn:NGW}, so we have that $|N_{G,r}(U\setminus B)|\geq \frac{\eps |U\setminus B|}{(\log n)^c}$. 
		Every vertex $v_i$ in $N_{G,r}(U\setminus B)$ has at least $r$ neighbours in $G$ in $U\setminus B$, and vertices of $U\setminus B$ must all have degree strictly less than $2t$ in $H$ (as they are not in $U'$). This implies that every $v_i \in N_{G,r}(U\setminus B)$ satisfies $v_i \in X$, since otherwise we could have added $v_i$, together with some $r$ of its neighbours, during the construction of $H$. Hence, $N_{G,r}(U\setminus B)\subseteq X$, and 
		\[
			|X|\geq |N_{G,r}(U\setminus B)|\geq \frac{\eps |U\setminus B|}{(\log n)^c}\geq  \frac{\eps|U|}{2(\log n)^c},
		\]
		as required.
	\end{proof}

To connect a given collection of vertex pairs by vertex-disjoint paths and thereby prove \Cref{lem:connecting}, we use the hypergraph version of Hall's theorem due to Aharoni and Haxell~\cite{aharoni2000hall}, an idea that also appears in \cite{bucic2022erdos}. Our application, however, differs slightly from that in \cite{bucic2022erdos}, since here our aim is to find \emph{vertex-disjoint} paths rather than edge-disjoint ones.

	\begin{theorem}[Aharoni--Haxell \cite{aharoni2000hall}] \label{thm:hall}
		Let $\ell,r \ge 1$ be integers, and let $\cH_1, \ldots, \cH_r$ be hypergraphs on the same vertex set with edges of size at most $\ell$. Suppose that, for every $I \subseteq [r]$, there is a matching in $\bigcup_{i \in I} \cH_i$ of size at least $\ell(|I|-1)$. 
		Then, there is an injective function $f : [r] \mapsto \bigcup_{i \in [r]} E(\cH_i)$ such that $f(i) \in E(\cH_i)$ for each $i \in [r]$ and $\{f(i) : i \in [r]\}$ is a matching with $r$ edges.
	\end{theorem}

Let us briefly explain how to translate the problem of finding vertex-disjoint paths 
into a hypergraph matching problem.  
Given a graph $G$ and distinct vertices $x_1,\dots,x_r,y_1,\dots,y_r$, fix 
$\ell\in\mathbb N$. For each $i\in[r]$, let $\cH_i$ be the hypergraph with vertex set 
$V$ whose edges correspond to the sets of internal vertices of all $(x_i,y_i)$-paths 
in $G$ of length at most $\ell$ whose internal vertices lie in $V$. If there exists an injective function 
$f:[r]\to\bigcup_{i\in[r]}E(\cH_i)$ such that $f(i)\in E(\cH_i)$ for each $i$ and the 
family $\{f(i):i\in[r]\}$ forms a matching, then for each $i$ we obtain a 
corresponding $(x_i,y_i)$-path $P_i$. Since the chosen hyperedges are disjoint, the 
internal vertex sets of the paths $P_1,\dots,P_r$ are pairwise disjoint, and hence 
the paths are vertex-disjoint.  Therefore, to find vertex-disjoint $(x_i,y_i)$-paths, it suffices to verify the 
conditions of \Cref{thm:hall} for the hypergraphs $\cH_1,\dots,\cH_r$.

	\subsection{Expansion into a random vertex set} \label{subsec:random-step-one}

		The following lemma shows that, given a sufficiently robust expander $G$ and a random vertex set $V\subseteq V(G)$, from any sufficiently large set $U$ one can, with high probability, reach more than half of the vertices of $V$ by short paths through $V$, while avoiding a given small set of vertices $Z$. This lemma is a variant of Lemma 17 in \cite{bucic2022erdos}. The main difference is that instead of a forbidden set of edges we have a forbidden set of vertices. Moreover, we have an additional parameter $c$ that controls the rate of expansion of our expanders and another parameter $q$ that controls the sampling probability for the random set $V$.

		\begin{lemma} \label{lem:expansion-random}
			Let $2^{-9} \le \eps < 1$, $0 < q < 1$, $c > 0$, $s \ge \frac{100 (\log n)^{2c+6}}{q^3}$ and let $n$ be large. Suppose that $G$ is an $n$-vertex $(\eps, c, s)$-expander and let $U, Z \subseteq V(G)$ be sets satisfying $|U| \ge \frac{(\log n)^{4c+9}}{q^5}$ and $|Z| \le \frac{\eps q |U|}{10^7(\log n)^{c}}$. Let $V$ be a random subset of $V(G)$, obtained by including each vertex independently with probability $q$.
			Then, with probability at least $1 - \exp\left(-\Omega\left(\frac{q^5|U|}{(\log n)^{4c+8}}\right)\right)$,
			\begin{equation*}
				\left| B^{(\log n)^{c+2}}_{G[V']}(U \cap V') \right| > \frac{|V|}{2},
			\end{equation*}
			where $V' \coloneqq V \setminus Z$.
		\end{lemma}

		\begin{proof}
			Let $\ell= (\log n)^{c+2}$ and let $p$ be such that $1-(1-p)^{\ell}(1-\frac{q}{20})(1- 0.9 q)=q$, i.e., that $(1-p)^{\ell}=\frac{1-q}{(1-\frac{q}{20})(1- 0.9 q)}$, so that
			\begin{equation}\label{eqn:p15}
				p\ge \frac{q}{100 \ell}.
			\end{equation}

			To prove \Cref{lem:expansion-random}, we will reveal the vertices in $V$ in $\ell+2$ batches, using the so-called \emph{sprinkling method}. More precisely, let $V_1, \ldots, V_{\ell}, V^{*}, V^{**}$ be random sets, where for each $1 \le i \le \ell$, $V_i$ is obtained by including each vertex with probability $p$, independently, $V^{*}$ is obtained by including each vertex with probability $\frac{q}{20}$, independently, and $V^{**}$ is obtained by including each vertex with probability $0.9 q$, independently. Notice that each vertex belongs to $V_1 \cup \cdots \cup V_{\ell} \cup V^{*} \cup V^{**}$ with probability $q$. Therefore, the random set $V$ (defined by including each vertex independently with probability $q$) has the same distribution as the union $V_1 \cup \cdots \cup V_{\ell} \cup V^{*} \cup V^{**}$.

			Define $U^* = (U \cap V^*) \setminus Z$. Then, by the Chernoff bound, with probability $1-\exp\left(-\Omega(q|U|)\right)$, we have that $|U^*| \ge \frac{q}{25}|U| - |Z| \ge \frac{q}{50}|U|$; condition on this being the case.
			Define $B_0 \coloneqq U^*$ and, for $i \ge 1$, let $B_i$ be the set of vertices in $G$ that can be reached by a path in $G \setminus Z$ that starts in $U^*$, has length at most $i$, and its internal vertices are in $V_1 \cup \ldots \cup V_i$. Let us emphasise that $B_i$ is only required to be disjoint from $Z$, and need not be a subset of $V_i$. Notice that $B_i \subseteq B_{i+1}$ for every $i \ge 0$, implying that for all $i \ge 0$, we have $$|B_i| \ge |U^*| \ge \frac{q}{50}|U|.$$

			\begin{claim}\label{claim:Bi-expand}
				For each $1\le i\le \ell-1$, if $|U^*| \ge \frac{q}{50}|U|$ and $|B_i| \le \frac{2}{3}n$, then, with probability at least $1 - \exp\left(-\Omega\left( \frac{q^5}{\ell^4}|U| \right)\right)$,
				\begin{equation*}
					\left|B_{i+1} \setminus B_i\right| \ge \frac{\eps |B_i|}{10^5 (\log n)^c}.
				\end{equation*}
			\end{claim}
			\begin{proof}
				Notice that a vertex in $N(B_i) \setminus Z$ is in $B_{i+1} \setminus B_i$ if at least one of its neighbours in $B_i$ is sampled into $V_{i+1}$.

                Consider the two possible outcomes obtained by applying \Cref{prop:stars} with $B_i$ playing the role of $U$ and with $r= \frac{\ell}{q}$ and $t = \left(\frac{\ell}{q}\right)^2 \cdot \frac{5}{(\log n)^{c}}$. (Note that the lemma indeed applies because $r,t \ge \ell \ge (\log n)^{2}$ and $s \ge \frac{100 (\log n)^{2c+6}}{q^3} = 20rt$.)

				Suppose that the first outcome of \Cref{prop:stars} holds, so there are $\frac{|B_i|}{10r}$ pairwise vertex-disjoint stars of size $t$ with centres in $B_i$ and leaves in $N(B_i)$. 
				By the Chernoff bound, with probability at least $1 - \exp\left(-\Omega\left(\frac{p|B_i|}{10r}\right)\right) = 1 - \exp\left(-\Omega\left(\frac{q^3}{\ell^2}|U|\right)\right)$, at least $\frac{p|B_i|}{20r}$ centres of these stars are included in $V_{i+1}$, implying that
				\begin{align*}
					\left|B_{i+1} \setminus B_i\right| 
					\ge \frac{pt|B_i|}{20r} - |Z|
					& \ge \frac{|B_i|}{20} \cdot \frac{q}{100\ell} \cdot \frac{5\ell^2}{q^2 (\log n)^{c}} \cdot \frac{q}{\ell} - |Z| \\
					& \ge \frac{|B_i|}{400 (\log n)^{c}}- |Z|
					\ge \frac{\eps|B_i|}{10^5 (\log n)^{c}},
				\end{align*}
				using that $|Z| \le \frac{\eps q |U|}{10^7 (\log n)^{c}} \le \frac{ |B_i|}{10^5 (\log n)^{c}}$ (since  $|B_i| \ge \frac{q}{50}|U|$ and $\eps \le 1$).

				Now suppose that the second outcome of \Cref{prop:stars} holds, so there is a bipartite subgraph $H \subseteq G$ with parts $B_i$ and $X \subseteq V(G) \setminus B_i$, with $|X| \ge \frac{\eps|B_i|}{2 (\log n)^{c}}$, such that vertices in $X$ have degree at least $r$ in $H$ while vertices in $B_i$ have degree at most $2t$ in $H$.
				Let $Y$ be the set of vertices in $X$ that do not have an $H$-neighbour in $V_{i+1}$. Note that $\Ex |Y| \le |X| (1 - p)^r \le |X| e^{-pr} \le 0.999 |X|$. Note also that $|Y|$ is $2t$-Lipschitz, since the outcome of the sampling of any single vertex in $B_i$ affects the outcome of at most $2t$ vertices in $X$. Thus, by \Cref{lem:concentration},
				\begin{align*}
					\Pr\left(|Y| > 0.9991|X|\right) \le \Pr\left(|Y| > \Ex |Y| + \frac{|X|}{10^4}\right)
					& \le 2\exp\left(-\frac{|X|^2}{2 \cdot 10^8 (2t)^2|B_i|}\right) \\
					& = \exp\left(-\Omega\left(\frac{q|U|}{t^{2} (\log n)^{2c}}\right)\right) \\
					& = \exp\left(-\Omega\left( \frac{q^5}{\ell^4}|U| \right)\right).
				\end{align*}
				Hence, in this case, with probability at least $1 - \exp\left(-\Omega\left( \frac{q^5}{\ell^4}|U| \right)\right)$, we have $|B_{i+1} \setminus B_i| \ge 0.0009 |X| - |Z| \ge \frac{\eps|B_i|}{10^4(\log n)^c} - |Z| \ge \frac{\eps|B_i|}{10^5 (\log n)^{c}}$ since $|Z| \le \frac{\eps q |U|}{10^7(\log n)^{c}} \le \frac{\eps |B_i|}{10^5 (\log n)^{c}}$.
			\end{proof}

			By repeatedly applying Claim~\ref{claim:Bi-expand}, we obtain that with probability at least 
			\begin{equation*}
				1 - \exp\left(-\Omega(q|U|)\right) - \ell \cdot \exp\left(-\Omega\left(\frac{q^5}{\ell^4}|U|\right)\right) = 
				1 - \exp\left(-\Omega\left(\frac{q^5}{\ell^4}|U|\right)\right),
			\end{equation*}
			we have $|U^*| \ge \frac{q}{50}|U|$ and for all $i \in [\ell]$ such that $|B_i| \le \frac{2}{3}n$, the following holds.
			\begin{equation*}
				|B_i| \ge \left(1 + \frac{\eps}{10^5 (\log n)^c}\right)^i |U^*|.
			\end{equation*}
			This implies that $|B_{\ell}| \ge \frac{2}{3}n$ with probability at least $1 - \exp\left(-\Omega\left(\frac{q^5}{\ell^4}|U|\right)\right)$.

			To complete the proof, note that any vertex in $B_{\ell}$ that gets sampled into $V^{**}$ (or that is already in $V_1 \cup \ldots \cup V_{\ell} \cup V^*$) is in the set $B' \coloneqq B_{G[V']}^{(\log n)^{c+2}}(U \cap V')$. 
   Hence, conditioning on the event that $|B_{\ell}| \ge \frac{2}{3}n$,
   by the Chernoff bound, at least $\frac{17}{30} qn$ vertices of $B_{\ell}$ are in $V$, and $|V| \le \frac{31}{30}qn$ with probability at least $1 - \exp({-\Omega(qn)})$. This implies that $|B'| > \frac{1}{2}|V|$ with probability at least $1 - \exp\left(-\Omega\left(\frac{q^5}{\ell^4}|U|\right)\right) - \exp({-\Omega(qn)}) = 1 - \exp\left(-\Omega\left(\frac{q^5}{\ell^4}|U|\right)\right)$, as required.
		\end{proof}

\Cref{lem:expansion-random} showed that if we choose $V\subseteq V(G)$ by including each vertex independently with probability $q$ in an $n$-vertex $(\varepsilon,c, s)$-expander $G$, then for any \emph{fixed} vertex set $U$ and any small vertex set $Z$, we can, with high probability, reach more than half of the vertices of $V$ by short paths through $V$ in $G-Z$. We now aim to show that an analogous property holds simultaneously for all such vertex sets $U$ (of size at most $2n/3$). However, we cannot apply a union bound directly, since the success probability in \Cref{lem:expansion-random} is not strong enough for this purpose.

To overcome this, we prove the following corollary by boosting the probability that a fixed set $U$ expands into $V$. This allows us to apply a union bound and conclude that all vertex sets $U$ (of size at most $2n/3$) expand into $V$, at the cost of imposing a slightly smaller upper bound on $|Z|$. This argument is a variant of Lemma~19 in~\cite{bucic2022erdos}. 

\begin{corollary}\label{cor:expansion-random}
			Let $2^{-9} \le \eps < 1$, $c > 0$, $0 < q < 1$, $s \ge \frac{2 (\log n)^{9c+21}}{q^{10}}$, and let $n$ be large. Suppose that $G$ is an $n$-vertex $(\eps, c, s)$-expander, and let $V$ be a random subset of $V(G)$, obtained by including each vertex independently with probability at least $q$. Then, with probability at least $1 - \frac{1}{n^2}$, for every $U, Z \subseteq V(G)$ with $|U| \le \frac{2}{3}n$ and $|Z| \le \frac{q^5|U|}{(\log n)^{5c+11}}$, we have 
			\begin{equation} \label{eqn:NU-expand}
				\left|B^{(\log n)^{c+2}}_{G[V']}\left(N(U) \cap V'\right)\right| > \frac{|V|}{2},
			\end{equation}
			where $V' \coloneqq V \setminus Z$.
		\end{corollary}

		\begin{proof}
			We say that a set $U \subseteq V(G)$ \emph{expands well} if $|N(U)| \ge \frac{|U| (\log n)^{4c+10}}{q^5}$. Given a (non-empty) set $U$ which expands well, and a set $Z$ with $|Z| \le |U|$, \Cref{lem:expansion-random} (applied with $N(U)$ playing the role of $U$) shows that \eqref{eqn:NU-expand} holds, with probability at least $1 - \exp\left(-\Omega\left(|U|(\log n)^{2}\right)\right)$.

			By the union bound, it follows that the probability that \eqref{eqn:NU-expand} fails for some pair $(U, Z)$, where $U$ expands well and $|Z| \le |U|$, is at most
			\begin{align*}
				\sum_{u = 1}^n (u+1) n^{2u} \exp\left(-\Omega\left(u(\log n)^{2}\right)\right)
				& \le \sum_{u = 1}^n \exp\left(3u \log n - \Omega\left(u(\log n)^2\right)\right) \\
				& = \sum_{u = 1}^n \exp\left(-\Omega\left(u (\log n)^2\right)\right) \le \frac{1}{n^2}.
			\end{align*}
			For the rest of the proof of \Cref{cor:expansion-random} we condition on the event that \eqref{eqn:NU-expand} holds for all pairs $(U, Z)$ where $U$ expands well and $|Z| \le |U|$. 
			From this, we will deduce that \eqref{eqn:NU-expand} holds for all pairs $(U, Z)$, where $U$ need not expand well, $|U| \le \frac{2}{3}n$, and $|Z| \le \frac{q^5|U|}{(\log n)^{5c+11}}$. Fix such a pair of sets $(U, Z)$.

			Write $d \coloneqq \frac{ (\log n)^{5c+11}}{q^5}$. By \Cref{prop:strong-expansion}, either $|N(U)| \ge \frac{s|U|}{2d}$, or $|N_d(U)| \ge \frac{\eps |U|}{(\log n)^{c}}$ (using the notation $N_d(U)$ introduced before the proposition). Notice that the first outcome implies that $U$ expands well (using the assumption on $s$ and the definition of $d$) and so we already know that \eqref{eqn:NU-expand} holds. Hence, we may assume that the second outcome holds, and write $W \coloneqq N_d(U)$. Let $U'$ be a subset of $U$ of size $\frac{|U|}{d}$, chosen uniformly at random. For a fixed $w \in W$, the probability that $w$ has no neighbours in $U'$ is at most
			\begin{equation*}
				\frac{\binom{|U|-d}{|U|/d}}{\binom{|U|}{|U|/d}} \le \left(\frac{|U|-d}{|U|}\right)^{|U|/d} \le e^{-1}.
			\end{equation*}
			It follows that $\Ex[|W \cap N(U')|] \ge (1 - e^{-1})|W| \ge \frac{\eps|U|}{2(\log n)^c} \ge \frac{|U'|(\log n)^{4c+10}}{q^5}$. In particular, there is a subset $U' \subseteq U$ of size $\frac{|U|}{d}$ with $|N(U')| \ge \frac{|U'| (\log n)^{4c+10}}{q^5}$. Since $U'$ expands well, and $|Z| \le \frac{q^5|U|}{(\log n)^{5c+11}} = |U'|$, \eqref{eqn:NU-expand} holds for $(U', Z)$ and thus \eqref{eqn:NU-expand} also holds for $(U, Z)$, completing the proof of the corollary. 
		\end{proof}	

	\subsection{A path connection through a random set} \label{subsec:expansion-step-two}

		For proving \Cref{lem:connecting}, our ultimate goal is to connect $r$ given pairs of vertices through a random set $V$ using internally vertex-disjoint paths, assuming that any subset of these vertices expands well. 
        The following lemma shows that at least one such pair of vertices can be connected while avoiding a small set of forbidden vertices. This will play a key role in the proof of \Cref{lem:connecting} in the next subsection, where we combine it with the hypergraph version of Hall's theorem (\Cref{thm:hall}) to connect several pairs simultaneously.
        
        \Cref{lem:hall-condition} is a variant of Proposition 8 in \cite{bucic2022erdos}, where, in addition to forbidding vertices instead of edges, we impose an expansion property on the vertices we wish to connect, and we use additional parameters $c$ and $q$ that control the rate of expansion of our expanders and the sampling probability of the random set $V$, respectively. 

		\begin{lemma} \label{lem:hall-condition}
			Let $2^{-9} \le \eps < 1$, $c > 0$, $0 < q < 1$, $s \ge \frac{2 (\log n)^{9c+21}}{q^{10}}$, and let $n$ be large. Suppose that $G$ is an $n$-vertex $(\eps, c, s)$-expander, and let $V$ be a random subset of $V(G)$, obtained by including each vertex independently with probability $q$.
			Then, with probability at least $1 - \frac{1}{n}$, the following holds for every $r$: If $x_1, \ldots, x_r, y_1, \ldots, y_r$ are distinct vertices, satisfying $|N(X)| \ge \frac{100 (\log n)^{7c+19}}{q^6}|X|$ for every subset $X \subseteq \{x_1, \ldots, x_r, y_1, \ldots, y_r\}$, and $Z$ is a set of size at most $2r(\log n)^{2c+8}$ which is disjoint from $\{x_1, \ldots, x_r, y_1, \ldots, y_r\}$, then for some $i \in [r]$ there is an $(x_i, y_i)$-path in $G$ whose internal vertices are in $V \setminus Z$ and whose length is at most $(\log n)^{c+4}$.
		\end{lemma}		

		\begin{proof}
			Let $V' \coloneqq V \setminus Z$. It is easy to see that by \Cref{cor:expansion-random} and the Chernoff bound (Theorem~\ref{chernoff}) together with a union bound, the following three properties hold simultaneously with probability at least $1 - \frac{1}{n}$. 
            
			\begin{enumerate}[label = \rm(T\arabic*)]
				\item \label{T1}
					For every $U \subseteq V(G)$ satisfying $|U| \le \frac{2n}{3}$ and $|Z| \le \frac{q^5|U|}{(\log n)^{5c+11}}$, \eqref{eqn:NU-expand} holds.
				\item \label{T2}
					We have $|V| \ge \frac{qn}{2}$.
				\item \label{T3}
					For every
					$X \subseteq \{x_1, \ldots, x_r, y_1, \ldots, y_r\}$ with $|X| \ge \frac{r}{2}$, we have $|N(X) \cap V'| \ge \frac{48 (\log n)^{7c+19}}{q^5} |X|$.
		      \end{enumerate}

              Indeed, by \Cref{cor:expansion-random}, \ref{T1} holds with probability at least $1 - \frac{1}{n^2}$. Moreover, \ref{T2} holds with probability at least $1 - \exp(-\Omega(qn)) \ge 1 - \frac{1}{n^2}$ by the Chernoff bound. To see why \ref{T3} holds, note that for a given $X$, the probability that $|N(X) \cap V| \ge q|N(X)|/2 \ge \frac{50 (\log n)^{7c+19}}{q^5}|X|$ is at least $1 - \exp(-\Omega(q|N(X)|)) \ge 1 - \exp(-\Omega((\log n)^{7c+19}/q^5)|X|)$. Hence by a union bound, for all $X \subseteq \{x_1, \ldots, x_r, y_1, \ldots, y_r\}$ with $|X| \ge r/2$, we have  $|N(X) \cap V'| \ge \frac{q|N(X)|}{2}-|Z| \ge \frac{50 (\log n)^{7c+19}}{q^5}|X| - |Z| \ge \frac{48 (\log n)^{7c+19}}{q^5}|X|$ with probability at least $1 - \frac{1}{n^2}$.

			Fix an outcome of $V$ such that the properties \ref{T1}--\ref{T3} hold simultaneously.
			Write $\ell \coloneqq (\log n)^{c+2}+1$.
			For a set of vertices $X$ and integer $d \ge 1$, define $R^d(X) \coloneqq B^d_{G[V']}(N(X) \cap V')$. Let $X_1$ be the set of vertices $x$ in $\{x_1, \ldots, x_r\}$ satisfying $|R^{\ell \log n}(x)| \le \frac{|V|}{2}$. 
			\begin{claim}
            \label{claim:X1atmostr2}
				$|X_1| < \frac{r}{2}$.
			\end{claim}
			\begin{proof}
				Suppose that $|X_1| \ge \frac{r}{2}$. We will first show that there is a sequence $(X_i)_{i \ge 1}$, such that for every $i \ge 1$, $X_{i+1} \subseteq X_{i}$, $|X_{i+1}| \le \max\{1, \frac{|X_i|}{2}\}$, and 
				\begin{equation} \label{eqn:R}
					\left|R^{i\ell}(X_i)\right| > \frac{|V|}{2}. 
				\end{equation}

				Indeed, notice that $|N(X_1) \cap V'| \ge \frac{48 (\log n)^{7c+19}}{q^5} |X_1| \ge \frac{r}{2} \cdot \frac{4 (\log n)^{7c+19}}{q^5} \ge \frac{(\log n)^{5c+11}}{q^5}|Z|$, by \ref{T3} and since $|X_1| \ge \frac{r}{2}$ and $|Z| \le 2r (\log n)^{2c+8}$. Thus, by \ref{T1}, $|R^{(\log n)^{c+2}}(N(X_1) \cap V')| > \frac{|V|}{2}$, showing that $|R^{\ell}(X_1)| > \frac{|V|}{2}$, proving that \eqref{eqn:R} holds for $i = 1$. (Notice that \eqref{eqn:NU-expand} is only applicable here if $|N(X_1) \cap V'| \le \frac{2n}{3}$, but if this fails then \eqref{eqn:R} holds trivially for $i = 1$.)	

				Now suppose that $X_1 \supseteq \ldots \supseteq X_i$ is a sequence, such that for all $j \in [i]$, $|X_{j}| \le \max\{1, \frac{|X_{j-1}|}{2}\}$ and \eqref{eqn:R} holds. If $|X_i| = 1$ we take $X_{i+1} = X_i$ (which clearly satisfies the required properties). Otherwise, by partitioning $X_i$ into at most three sets of size at most $\frac{|X_i|}{2}$, there is a subset $X_{i+1} \subseteq X_i$ of size at most $\frac{|X_i|}{2}$ satisfying $|R^{i\ell}(X_{i+1})| \ge \frac{|V|}{6}$. 
				We will show that $X_{i+1}$ satisfies \eqref{eqn:R}. Indeed, consider the set $U \coloneqq R^{i\ell}(X_{i+1})$. 
				Then, \eqref{eqn:R} trivially holds if $|U| \ge \frac{2n}{3}$, so suppose otherwise.
				We have $|U| \ge \frac{|V|}{6} \ge \frac{qn}{12} \ge \frac{q}{12} \cdot r \cdot \frac{100 (\log n)^{7c+19}}{q^6} \ge \frac{(\log n)^{5c+11}}{q^5}|Z|$, where the third inequality follows from the assumption that $|N(\{x_1, \ldots, x_r\})| \ge r \cdot \frac{100(\log n)^{7c+19}}{q^6}$ (and by the trivial inequality $n \ge |N(\{x_1, \ldots, x_r\})|$). Thus, by \ref{T1},
				\begin{equation*}
					\left|B_{G[V']}^{(\log n)^{c+2}}\left(N(U) \cap V'\right)\right|
					> \frac{|V|}{2}.
				\end{equation*}
				Notice that
				\begin{equation*}
					B^{(\log n)^{c+2}}_{G[V']}\left(N(U) \cap V'\right) 
					\subseteq B_{G[V']}^{\ell}(U)
					= R^{(i+1) \ell}(X_{i+1}),
				\end{equation*}
				so $X_{i+1}$ has the required properties.
				This completes the proof of the existence of a sequence $(X_i)_{i \ge 1}$ with the desired properties. 

				Fix such a sequence, and take $i \coloneqq \log n$. Then, $|X_i| \le \max\{1, 2^{-\log n}|X_1|\} = 1$. This means $|R^{\ell \log n}(x)| > \frac{|V|}{2}$ for the single vertex $x$ in $X_i$, contradicting the choice of $X_1$, and completing the proof of the claim.
			\end{proof}
			Take $Y_1$ to be the set of vertices $y \in \{y_1, \ldots, y_r\}$ satisfying $|R^{\ell \log n}(y)| \le \frac{|V|}{2}$. Then, analogously to \Cref{claim:X1atmostr2}, $|Y_1| < \frac{r}{2}$.
			Hence, there exists $i \in [r]$ such that $|R^{\ell \log n}(x_i)|, |R^{\ell \log n}(y_i)| > \frac{|V|}{2}$, showing that, with probability at least $1 - \frac{1}{n}$, there is an $(x_i,y_i)$-path of length at most $2\ell \log n + 2 \le (\log n)^{c+4}$, with internal vertices in $V'$, completing the proof of \Cref{lem:hall-condition}.
		\end{proof}

	\subsection{Many vertex-disjoint path connections through a random set} \label{subsec:expansion-step-three}

		Finally, we are ready to prove \Cref{lem:connecting}, which states that given a sufficiently robust expander $G$ and a random subset of vertices $V$, then with high probability for any set $\{x_1, \ldots, x_r, y_1, \ldots, y_r\}$ of vertices outside of $V$, whose subsets all expand well, the pairs $(x_i, y_i)$, $i \in [r]$, can be connected using vertex-disjoint paths through $V$.

		\begin{proof}[ of \Cref{lem:connecting}]
			Fix an outcome of $V$ for which the conclusion of \Cref{lem:hall-condition} holds. For each $i \in [r]$, let $\cH_i$ be the hypergraph on the vertex set $V$ where each edge is the set of internal vertices of $P$, for all $(x_i, y_i)$-paths $P$ of length at most $\ell \coloneqq (\log n)^{c+4}$ with internal vertices in $V$.

			We will apply \Cref{thm:hall}. Fix a subset $I \subseteq [r]$. We wish to show that there is a matching of size at least $\ell(|I|-1)$ in $\cH' \coloneqq \bigcup_{i \in I} \cH_i$. Suppose that no such matching exists. Let $\cM'$ be a maximal matching in $\cH'$, and let $Z$ be the set of vertices in $\cM'$. Then, $|\cM'| < \ell(|I|-1)$, every edge in $\cH'$ intersects $Z$, and $|Z| \le \ell^2(|I|-1) \le |I|(\log n)^{2c+8}$. Hence, by applying \Cref{lem:hall-condition} (with $I$ playing the role of $[r]$) we obtain that for some $i \in I$ there is an $(x_i, y_i)$-path $P$ of length at most $\ell$ whose internal vertices are in $V \setminus Z$. But this means that $V(P) \setminus \{x_i, y_i\}$ is an edge in $\cH'$ that does not intersect $Z$, a contradiction.

			Therefore, the assumptions in \Cref{thm:hall} are satisfied, showing that there is a matching $\cM$ of size $r$ in $\bigcup_{i \in [r]}\cH_i$, such that the $i$-th edge of the matching is in $\cH_i$. Let $P_i$ be the path corresponding to the $i$-th edge in $\cM$. Then, for each $i \in [r]$, $P_i$ is an $(x_i, y_i)$-path with internal vertices in $V$. The internal vertex sets of the paths $P_1, \ldots, P_r$ are pairwise disjoint and are disjoint from $\{x_1, \ldots, x_r, y_1, \ldots, y_r\}$, as it is assumed that the vertices in this set are not contained in $V$, thereby proving the lemma.
		\end{proof}

\section{Finding an almost-spanning \texorpdfstring{$F$}{F}-subdivision in a nearly regular expander}
\label{sec:almostspanningFsubdivisioninexpander}

In this section we prove our second key lemma (Lemma~\ref{lem:almostspannningtopological}) which shows that any sufficiently regular expander contains an \emph{almost-spanning} subdivision of any given (non-empty) graph $F$.

\begin{lemma}
	\label{lem:almostspannningtopological}
	Let $c, \eps, d, s, n$ be parameters such that $c > 0$ is fixed, $n$ is sufficiently large, $0 < \eps \le \frac{1}{(\log n)^{4}}$, $d \ge (\log n)^{10c+51}$ and $s \ge \frac{d}{4 (\log n)^{c}}$.
	Let $F$ be a non-empty graph on at most $\frac{\sqrt{d}}{(\log n)^{2}}$ vertices, and let $H$ be an $n$-vertex $(\frac{1}{8}, c, s)$-expander such that $\Delta(H) \le d$, $d(H) \ge d(1 - \eps)$ and $\delta(H) \ge \frac{d(H)}{2}$. Then, $H$ contains a subdivision of $F$ covering all but at most $\frac{n}{\log n}$ vertices of $H$.
\end{lemma}

Lemma~\ref{lem:almostspannningtopological} implies that every sufficiently regular expander contains a nearly Hamilton cycle. Indeed, by letting $F$ be a triangle in the lemma, we obtain that $H$ contains a cycle $C$ covering all but at most $\frac{n}{\log n}$ vertices of $H$. 
This statement is of independent interest and is likely to have other applications. 

For instance, we can use it to prove that nearly all vertices of any $d$-regular graph of order $n$, with $d \ge (\log n)^{130}$, can be covered by at most $\frac{n}{d+1}$ vertex-disjoint cycles.    
Recently, Montgomery, M\"uyesser, Pokrovskiy and Sudakov \cite{montgomery2024approximate} proved a similar result, showing that nearly all vertices of every $d$-regular graph of order $n$ can be covered by at most $\frac{n}{d+1}$ vertex-disjoint paths, thus making progress on a conjecture of Magnant and Martin~\cite{magnant2009note} (which states that all vertices can be covered by the same number of vertex-disjoint paths). Their result is stronger than ours in that it applies for all $d$, but ours is also slightly stronger in that it finds vertex-disjoint cycles rather than paths.

Additionally, combining \Cref{lem:almostspannningtopological} with \Cref{cor:findexpander}, 
we recover a recent result of Dragani\'c, Methuku, Munh\'a Correia, and Sudakov~\cite{draganic2023cycles} 
on finding a cycle with many chords, albeit with a slightly larger logarithmic factor 
in the number of edges required to guarantee such a cycle. We elaborate on both applications in \Cref{sec:application:chords}.
 
\noindent In the rest of this section, we prove Lemma~\ref{lem:almostspannningtopological}. Let us first provide a brief outline of the structure of the proof. We start by taking a random partition of the vertex set of our $(\frac{1}{8}, c, s)$-expander $H$ into sets $V_0, X_1, \ldots, X_t, R$, where the sizes of these random sets are chosen carefully in Section~\ref{subsec:setupparameters}. In particular, we ensure that $X \coloneqq \bigcup_{i = 1}^t X_i$ contains nearly all vertices of $H$ and the size of $V_0$ is sufficiently larger than the sizes of the sets $X_1, \ldots, X_t, R$. In Section~\ref{subsec:finding a linear forest}, we show that, with high probability, there exists a small collection of paths $P_1, \dots, P_r$ contained in $X$, with leaves in $X_1 \cup X_t$, that covers nearly all vertices of $H$. Then, in Section~\ref{subsec:joiningthepaths}, we iteratively connect these paths through $V_0$ to construct a nearly Hamilton path $P$. More precisely, in each iteration, we join a constant proportion of the paths $P_1, \dots, P_r$ via $V_0$ using the following strategy: first, we greedily connect as many paths as possible through $V_0$ using vertex-disjoint paths of length two. When this is no longer possible, we show that there is a small subset $S$ of the leaves of $P_1, \dots, P_r$ that expands well into $V_0$, allowing us to apply \Cref{lem:connecting} to connect the paths whose leaves lie in $S$ via $V_0$ (see \Cref{claim:connectingverticesofY}). Finally, in Section~\ref{sec:findingFsubdivisioncontainingP}, we find a subdivision of $F$ in $H$ that contains the nearly Hamilton path $P$, yielding the desired almost spanning $F$-subdivision in $H$ (see Figure~\ref{fig:almostspanningsubdivision}).

\begin{figure}[h]
    \centering
	\includegraphics[width= \textwidth]{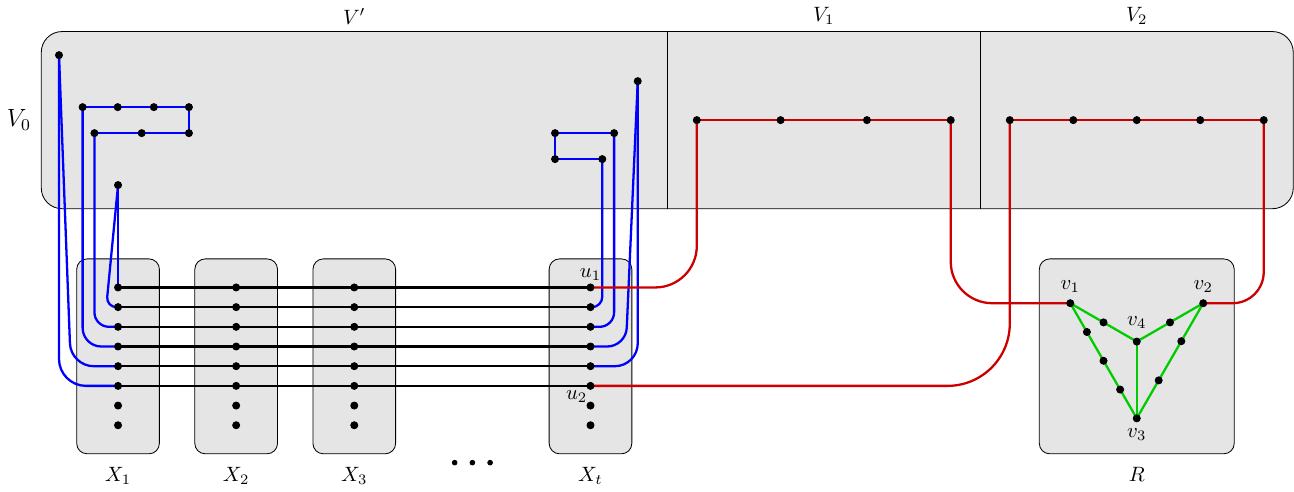}
    \caption{The figure shows how to construct an almost spanning $F$-subdivision in $H$ (when $F = K_4$).}
\label{fig:almostspanningsubdivision}
\end{figure}

\subsection{Setting up the parameters}
\label{subsec:setupparameters}

Throughout the section, we work with parameters $c,\eps,d,s,n$ satisfying the conditions
from Lemma~\ref{lem:almostspannningtopological}, namely
\[
c>0,\quad n \text{ is sufficiently large},\quad 
0<\eps \le (\log n)^{-4},\quad 
d \ge (\log n)^{10c+51},\quad \text{ and } 
\quad s \ge \frac{d}{4(\log n)^c}.
\]

We let $H$ be an $n$-vertex $(\frac{1}{8}, c, s)$-expander with $s \ge \frac{d}{4 (\log n)^c}$ such that $\Delta(H) \le d$, $d(H) \ge d(1 - \eps)$ and $\delta(H) \ge \frac{d(H)}{2}$. Moreover, throughout this section, we let

\begin{equation*}
	q_1 \coloneqq \frac{1}{3\log n}, \quad 
	q_2 \coloneqq \frac{6}{(\log n)^3}, \quad 
	\eps_0 \coloneqq \frac{1}{(\log n)^4}, \quad
	t \coloneqq \frac{1}{6} (\log n)^3 - \frac{1}{18} (\log n)^2-1.
\end{equation*}

We make some quick observations that will be used in the rest of the proof. Note that $d(H)  \ge d(1 - \eps) \ge d(1 - \eps_0)$ since $\eps \le \frac{1}{(\log n)^4} = \eps_0 $.
Also note that $q_1 + (t+1)q_2 = 1$, that $q_2 \le \frac{q_1}{4000 \log n}$ and that $t \le \frac{1}{q_2}$. Moreover, $d \ge (\log n)^{10 c + 51} \ge \frac{8(\log n)^{10c +21}}{(2q_2)^{10}}$, $\eps_0 \ge 2d^{-1/4}$, and $\frac{\eps_0}{q_2} = \frac{1}{6\log n}$.

Let us define a random partition $\{V_0, X_1, \ldots, X_t, R\}$ of $V(H)$ as follows. 
Independently, for each vertex $v\in V(H)$, the probability that $v$ is included in $V_0$ is $q_1$, the probability that $v$ is included in $X_i$  is $q_2$ for each $i \in [t]$, and the probability that $v$ is included in $R$ is also $q_2$. (Notice that these probabilities sum to $q_1 + (t+1)q_2 = 1$.) For convenience, let $X \coloneqq \bigcup_{i = 1}^t X_i$, so that $V(H) = V_0 \cup X \cup R$. 

\subsection{Finding vertex-disjoint paths \texorpdfstring{$P_1, \ldots, P_r$}{P1, \ldots, Pr} covering most of the vertices of \texorpdfstring{$X$}{X}}
\label{subsec:finding a linear forest}

	In the next claim, we show that, with high probability, there exists a small collection of paths in $X$, with endpoints in $X_1\cup X_t$, that together cover almost all vertices of $H$. We achieve this by finding large matchings in the bipartite subgraphs $H[X_i,X_{i+1}]$ for each $1\le i\le t-1$, which are then combined to form a large linear forest.  The idea of ``gluing'' together matchings to construct a large linear forest is a well-known technique. Variants of this approach appear, for example, in the work of Kelmans, Mubayi and Sudakov~\cite{KelmansMubayiSudakov2001} on tree packings as well as more recent works including the random Hall--Paige paper of M{\"u}yesser and Pokrovskiy~\cite{MuyesserPokrovskiy2025RandomHallPaige} and the approximate path decompositions of regular graphs by Montgomery, M{\"u}yesser, Pokrovskiy and Sudakov~\cite{montgomery2024approximate}.

    In Section~\ref{subsec:joiningthepaths}, we will join these paths to form a nearly Hamilton path in $H$.

	\begin{claim}
		\label{claim:findingpaths}
		With probability at least $1 - \exp \left(-\Omega(\log n)^2\right),$ there exists a collection of vertex-disjoint paths $P_1, \ldots, P_r$ satisfying the following properties.
		\begin{enumerate}[label = \rm(P\arabic*)]
			\item \label{itm:findpaths-1} 
				$\bigcup_{i \in [r]}V(P_i) \subseteq X$.
			\item \label{itm:findpaths-2}
				For $i \in [r]$, the leaves of $P_i$ are in $X_1 \cup X_t$.
			\item \label{itm:findpaths-3}
				$|X \setminus \bigcup_{i = 1 }^r V(P_i)| \le \frac{1}{2\log n} |X|$.
		\end{enumerate}
	\end{claim}

	\begin{proof}[ of claim]
		Let us first show that the following three properties hold simultaneously with probability $1 - \exp\left(-\Omega\left((\log n)^{2}\right)\right)$.
		\begin{enumerate}[label = \rm(X\arabic*)]
			\item \label{X1item}
				$e_H(X_i, X_{i+1}) \ge d(1 - \eps)n q_2^2 - d q_2^2 n^{2/3}$ for all $i \in [t-1]$. 
			\item \label{X2item}
				$|X_i| = (1 \pm n^{-1/3})q_2 n $ for each $i \in [t]$.
			\item \label{X3item}
				$\Delta(H[X_i, X_{i+1}]) \le q_2 d(1+d^{-1/3})$ for all $i \in [t-1]$.
		\end{enumerate}
		To show this, it suffices to prove that each of the properties \ref{X1item}--\ref{X3item} holds with probability $1 - \exp\left(-\Omega\left((\log n)^{2}\right)\right)$.

		For \ref{X1item}, note that for any $i \in [t-1]$, we have $\mathbb{E}[e_H(X_i, X_{i+1})] = e(H) \cdot 2q_2^2 \ge d(1 - \eps)n \cdot q_2^2$ since for any edge $uv \in E(H)$, the probability that $u \in X_i, v \in X_{i+1}$ or $v \in X_i, u \in X_{i+1}$ is $2q_2^2$. Now note that changing the outcome of whether a certain vertex $v \in V(H)$ belongs to $X_i$, $X_{i+1}$, or neither, changes $e_H(X_i, X_{i+1})$ by at most $\Delta(H) \le d$, so $e_H(X_i, X_{i+1})$ is $d$-Lipschitz. Hence, by Lemma~\ref{lem:concentration}, $$\Pr\left[e_H(X_i, X_{i+1}) < d(1 - \eps)n q_2^2 - d q_2^2 n^{2/3}\right] \le 2 \exp \left(-\Omega(q_2^4n^{1/3})\right).$$
		Therefore, by the union bound, since $t \le n$, \ref{X1item} holds with probability at least $1 - 2t \exp \left(-\Omega(q_2^4n^{1/3})\right) \ge 1 - \exp \left(-\Omega(\log n)^2\right)$, as desired. 

		For \ref{X2item}, note that we have $\mathbb{E}[|X_i|]= q_2 n$ for each $i \in [t]$. Then, by the Chernoff bound (Theorem~\ref{chernoff}), the probability that $|X_i| =  (1 \pm n^{-1/3})q_2 n$ is at least $1 - \exp \left(-\Omega(q_2 n^{1/3})\right)$. Again, by the union bound, since $t \le n$, \ref{X2item} holds with probability at least $1 - t \exp \left(-\Omega(q_2 n^{1/3})\right) \ge 1 - \exp \left(-\Omega(\log n)^2\right)$, as desired. 

		Finally, for \ref{X3item}, note that since $|N_H(v)| \le d$ for any $v \in V(H)$, we have $\mathbb{E}[|N_H(v) \cap X_i|] \le q_2 d$. Hence, by the Chernoff bound (Theorem~\ref{chernoff}), for any $v \in V(H)$, $|N_H(v) \cap X_i| \le q_2 d(1+d^{-1/3})$ with probability at least $ 1 - \exp \left(-\Omega(\log n)^2\right)$. Therefore, by the union bound, for every $v \in V(H)$ and $i \in [t]$, $|N_H(v) \cap X_i| \le q_2 d(1+d^{-1/3})$ with probability at least $1 - t n \exp \left(-\Omega(\log n)^2\right) \ge 1 - \exp \left(-\Omega(\log n)^2\right)$, implying that \ref{X3item} holds with the required probability. This shows that \ref{X1item}--\ref{X3item} hold simultaneously with probability $1 - \exp\left(-\Omega\left((\log n)^{2}\right)\right)$, as desired.

		We now assume that \ref{X1item}--\ref{X3item} hold, and use this to deduce that the paths $P_1, \ldots, P_r$ satisfying \ref{itm:findpaths-1}--\ref{itm:findpaths-3} exist with the required probability.
		By Vizing's theorem, for every $i \in [t-1]$, there is a matching $M_i$ in $H[X_i, X_{i+1}]$ satisfying
		\begin{align*}
			|M_i| 
			\ge \frac{e_H(X_i, X_{i+1})}{\Delta(H[X_i, X_{i+1}])+1} 
			& \ge \frac{d(1 - \eps)n q_2^2 - d q_2^2 n^{2/3}}{q_2 d(1+d^{-1/4})} \\[.5em]
			& \ge q_2n(1 - \eps-n^{-1/3})(1-d^{-1/4}) \\[.5em]
			& \ge \frac{|X_i|}{(1 + n^{-1/3})} \cdot (1 - \eps-n^{-1/3})(1-d^{-1/4}) \\[.5em]
			& \ge |X_i|(1 - n^{-1/3})(1 - \eps - n^{-1/3})(1 - d^{-1/4}) \\
			& \ge |X_i|(1 - \eps - 2n^{-1/3} - d^{-1/4}) \\
			& \ge |X_i|(1 - \eps_0 - 2d^{-1/4}) 
			\ge |X_i|(1 - 2\eps_0).
		\end{align*}

		Let $M'_1 \coloneqq M_1$, and for each $2 \le i \le t-1$, let $M'_i \subseteq M_i$ be the subset of edges of $M_i$ which are incident to an edge of $M'_{i-1}$. Then, for each $2 \le i \le t-1$, since $|M_i| \ge |X_i|(1 - 2\eps_0),$ we have $|M'_i| \ge |M'_{i-1}|-2\eps_0 |X_{i}|$. This implies that $|M'_{t-1}| \ge |M_{1}|-2\eps_0 (\sum_{j=2}^{t-1}|X_{j}|) \ge |X_{1}|-2\eps_0 (\sum_{j=1}^{t-1}|X_{j}|) \ge |X_{1}|-2\eps_0 |X|$. Now, since every edge of $M'_{i}$ is incident to an edge of $M'_{i-1}$ for every $2 \le i \le t-1$, it follows that each edge of $M'_{t-1}$ lies on a path from $X_1$ to $X_t$ in $M_1 \cup \ldots \cup M_{t-1}$ which contains exactly one vertex from each $X_j$ (for $j \in [t]$), such that the paths are vertex-disjoint. Let us denote these paths by $P_1, \ldots, P_r$, where $r \coloneqq |M'_{t-1}|$. Then, it is easy to see that \ref{itm:findpaths-1} and \ref{itm:findpaths-2} hold.
		For \ref{itm:findpaths-3}, notice that since \ref{X2item} holds, $t|X_1| \ge t(1 - n^{-1/3})q_2n = t(1 + n^{-1/3})q_2n \cdot \frac{1 - n^{-1/3}}{1 + n^{-1/3}} \ge |X|(1 - 2n^{-1/3}) \ge |X|(1 - \eps_0)$.
		Thus, $\sum_{i = 1}^r|V(P_i)| = rt \ge (|X_{1}|-2\eps_0 |X|)t \ge |X|(1 - 3 \eps_0 t)$. Hence, $|X \setminus \bigcup_{i = 1 }^r V(P_i)| \le 3 \eps_0 t |X| \le \frac{3 \eps_0}{q_2} |X| = \frac{1}{2\log n} |X|,$ as desired. This proves the claim.
	\end{proof}

\subsection{Iteratively connecting the paths \texorpdfstring{$P_1, \ldots, P_r$}{P1, \ldots, Pr} through \texorpdfstring{$V_0$}{V0} to construct a nearly Hamilton path}
\label{subsec:joiningthepaths}
    In this subsection, our aim is to construct a nearly Hamilton path in $H$ with high probability. 
	We do this in Claim~\ref{claim:almostspanningpath} by repeatedly applying the following claim which shows that if $V$ and $W$ are two random sets, with $W$ only slightly smaller than $V$, then with high probability one can find vertex-disjoint paths (with internal vertices in $V$) that connect a positive proportion of the vertices in any given subset of $W$. Note that, as discussed in Section~\ref{sec:proofsketch}, although the task becomes easier when $W$ is significantly smaller than $V$ (depending on $c$), it is crucial to our argument that $W$ is allowed to be only slightly smaller than $V$.

	\begin{claim}
		\label{claim:connectingverticesofY}
		Let $0 < p_1 , p_2 < 1$ such that $p_2 \le \frac{p_1}{100}$, let $d \ge \frac{8 (\log n)^{10c +21}}{p_2^{10}}$. Let $V, W \subseteq V(H)$ be disjoint random sets such that, independently, for each vertex $v \in V(H)$, the probability that $v$ is included in $V$ is $p_1$, and the probability that $v$ is included in $W$ is $p_2$. Then, with probability at least $1 - \frac{2}{n}$, for every $Y \subseteq W$ of size $k \ge 2$, there are at least $\frac{k}{10}$ pairwise vertex-disjoint paths of length at most $(\log n)^{c+4}$ whose internal vertices belong to $V$ and whose endpoints lie in $Y$. 
	\end{claim}

	\begin{proof}[ of claim]
		We claim that the following three properties hold simultaneously with probability at least $1 - \frac{2}{n}$.
		\begin{enumerate}[label=(SP\arabic*)]
			\item \label{claim7-1}
				For any vertex $v \in V$, we have $|N_H(v) \cap W| \le 2 p_2 d$.
			\item \label{claim7-2}
				For any vertex $v \in W$, we have $|N_H(v) \cap V| \ge \frac{p_1 d}{5}$.
			\item \label{claim7-3}
				If $x_1, \ldots, x_r, y_1, \ldots, y_r \in V(H) \setminus V$ are distinct vertices such that every subset $X \subseteq \{x_1, \ldots, x_r, y_1, \ldots, y_r\}$ satisfies $|N_H(X)| \ge \frac{100 (\log n)^{7c+19}}{p_1^6} |X|$, then there are pairwise vertex-disjoint paths $Q_1, \ldots, Q_r$ of length at most $(\log n)^{c+4}$ with internal vertices in $V$, such that $Q_i$ is a path joining $x_i$ and $y_i$.
		\end{enumerate}
		Indeed, since $\Delta(H) \le d$, for any vertex $v \in V$, we have $\mathbb{E}[|N_H(v) \cap W|] \le p_2 d$, so by the Chernoff bound (Theorem~\ref{chernoff}), we have $|N_H(v) \cap W| \le 2 p_2 d$ with probability at least $1 - \exp(-\Omega(p_2d)) \ge 1 - \exp(-\Omega(\log n)^2)$. Therefore, by the union bound, \ref{claim7-1} holds with probability at least $1 - \exp(-\Omega(\log n)^2)$.
		Since $\delta(H) \ge \frac{d(H)}{2} \ge \frac{d}{4}$, for any $v \in W$, we have $\mathbb{E}[|N_H(v) \cap V|] \ge \frac{p_1 d}{4}$, so by the Chernoff bound, we have $|N_H(v) \cap V| \ge \frac{p_1 d}{5}$ with probability at least $1 - \exp(-\Omega(p_1d)) \ge 1 - \exp(-\Omega(\log n)^2)$. Therefore, by the union bound, \ref{claim7-2} holds with probability at least $1 - \exp(-\Omega(\log n)^2)$.
		For \ref{claim7-3}, notice that $H$ is a $(\frac{1}{8}, c, s)$-expander with $s \ge \frac{d}{4 (\log n)^c} \ge \frac{2 (\log n)^{9c+21}}{p_2^{10}} \ge \frac{2 (\log n)^{9c+21}}{p_1^{10}}$, so we can apply \Cref{lem:connecting} with $H, p_1$ playing the roles of $G, q$ respectively. Thus, \ref{claim7-3} holds with probability at least $1 - \frac{1}{n}$.
		By the union bound, \ref{claim7-1}--\ref{claim7-3} hold simultaneously, with probability at least $1 - \frac{2}{n}$.

		In the rest of the proof of this claim, we assume that \ref{claim7-1}--\ref{claim7-3} hold and prove that for every $Y \subseteq W$ of size $k \ge 2$, there are at least $\frac{k}{10}$ pairwise vertex-disjoint paths of length at most $(\log n)^{c+4}$ whose internal vertices are in $V$ and whose leaves are in $Y$. Indeed, let $Y \subseteq W$ be a set of size $k \ge 2$.
		Let $\mathcal{C}$ be a maximal collection of vertex-disjoint paths in $H$ of the form $abc$ with $b \in V$ and $a, c \in Y$. If $|\mathcal C| \ge k/10$, then $\mathcal{C}$ is the collection of vertex-disjoint paths required by the claim. So we may suppose $|\mathcal C| < k/10$. 

		Let $Y' \subseteq Y$ be the set of vertices in $Y$ which are not contained in any of the paths in $\mathcal{C}$. Then, $|Y'| \ge |Y| - 2|\mathcal{C}| \ge |Y| - \frac{k}{5} = \frac{4k}{5}$. 
		Let $Y''$ be the set of vertices in $Y'$ which have at least $\frac{p_1 d}{30}$ neighbours (in $H$) in the set $S \coloneqq \{ b \mid abc \in \mathcal{C} \}$. Note that $S \subseteq V$. We claim that $|Y''| \le \frac{k}{10}$. Indeed,  for every $v \in V$, $|N_H(v) \cap W| \le 2 p_2 d$ by \ref{claim7-1} and $Y'' \subseteq W$, so we have $|N_H(v) \cap Y''| \le 2 p_2 d$. Thus, 
		\begin{equation*}
			|Y''|\cdot \frac{p_1 d}{30} \le e_H(Y'',S) \le |S| \cdot 2p_2 d = |\mathcal{C}|\cdot 2p_2 d \le \frac{k}{10} \cdot 2p_2 d.
		\end{equation*}
		Hence, using that $p_2 \le p_1/100$, we have
		$|Y''| \le 6 k \cdot \frac{p_2}{p_1} \le \frac{k}{10},$ as desired. 
  
  Let $Y^* \coloneqq Y' \setminus Y''$. Then, since $|Y'| \ge \frac{4k}{5}$ and $|Y''| \le \frac{k}{10}$, we have $|Y^*| \ge \frac{4k}{5} - \frac{k}{10} \ge \frac{7 k}{10}$, and for every vertex $v \in Y^*$, we have $|N_H(v) \cap S| < \frac{p_1 d}{30}$ by the choice of $Y^*$. Thus, for every $v \in Y^* \subseteq W$, we have $|N_H(v) \cap (V \setminus S)| \ge \frac{p_1 d}{5} - \frac{p_1 d}{30} = \frac{p_1d}{6}$ since $|N_H(v) \cap V| \ge \frac{p_1 d}{5}$ by \ref{claim7-2}. Moreover, by the maximality of the collection $\mathcal{C}$, for any two distinct vertices $u, v \in Y^*$, we have $(N_H(u) \cap (V \setminus S)) \cap (N_H(v) \cap (V \setminus S)) = \emptyset$. In particular, this implies that for any $Z \subseteq Y^*$, we have $|N_H(Z)| \ge |Z|\frac{p_1 d}{6}$. 

	Let $Y^*_{\mathrm{even}} \subseteq Y^*$ be a subset of size at least $|Y^*|-1$ such that $|Y^*_{\mathrm{even}}|$ is even.
		Note that for every $Z \subseteq Y^*_{\mathrm{even}}$, we have $|N_H(Z)| \ge |Z|\frac{p_1 d}{6} \ge |Z| \frac{100 (\log n)^{7c+19}}{p_1^6}$. 
		Thus, using \ref{claim7-3}, with $Y^*_{\mathrm{even}}$ playing the role of $\{x_1, \ldots, x_r, y_1, \ldots, y_r\}$, we obtain $r = \frac{|Y^*_{\mathrm{even}}|}{2} \ge \frac{|Y^*|-1}{2} \ge \frac{k}{10}$ pairwise vertex-disjoint paths of length at most $(\log n)^{c+4}$ whose internal vertices are in $V$, and whose leaves are in $Y^* \subseteq Y$. 
		This proves that the claim holds with probability at least $1 - \frac{2}{n}$, as desired.
	\end{proof}
 
    Recall that $\{V_0, X_1, \ldots, X_t, R\}$ is a random partition of $V(H)$, and $X = \bigcup_{i = 1}^t X_i$. In the next claim we repeatedly apply Claim~\ref{claim:connectingverticesofY} to join the paths $P_1, \ldots, P_r$ (guaranteed by \Cref{claim:findingpaths})  through the set $V_0$ to obtain a nearly Hamilton path in $H$ with leaves in $X_1 \cup X_t$.
 
	\begin{claim}
		\label{claim:almostspanningpath}
		Let $V'$ be a random subset of $V_0$ obtained by including each vertex of $V_0$ in $V'$ with probability $1/2$. 
		Then, with probability at least $1 - \frac{21 \log n}{n}$, there is a path $P$ in $H$ such that $V(P) \subseteq V' \cup X$, the leaves of $P$ are contained in $X_1 \cup X_t$, and $|X - V(P)| \le \frac{1}{2\log n}|X|$. 
	\end{claim}

	\begin{proof}
		Since every vertex $v \in V(H)$ belongs to $V_0$ with probability $q_1$, every vertex $v \in V(H)$ is included in $V'$ with probability $\frac{q_1}{2}$. Let $\ell \coloneqq 10 \log n$ and let $q_3 \coloneqq \frac{q_1}{2\ell}$. Note that since $q_2 \le \frac{q_1}{4000 \log n} = \frac{q_1}{400 \ell}$, we have $q_2 \le \frac{q_3}{200}$. 
		Let $V' \coloneqq \bigcup_{i = 1}^{\ell} V'_i$ be a random partition of $V'$ so that every vertex $v \in V(H)$ is included in $V'_i$ with probability $q_3$, for every $i \in [\ell]$. 
		In the rest of the proof of the claim, we condition on the existence of paths $P_1, \ldots, P_r$ satisfying \ref{itm:findpaths-1}--\ref{itm:findpaths-3} of \Cref{claim:findingpaths}, and on the conclusion of \Cref{claim:connectingverticesofY} holding with $V_i', X_1 \cup X_t, q_3, 2q_2$, playing the roles of $V, W, p_1, p_2$, respectively, for each $i \in [\ell]$. (Note that \Cref{claim:connectingverticesofY} is applied $\ell$ times in the latter statement using that $2q_2 \le \frac{q_3}{100}$ and $d \ge \frac{8 (\log n)^{10c+21}}{(2q_2)^{10}}$.) 
		Indeed, the former statement holds with probability at least $1 - \exp\left(-\Omega\left((\log n)^{2}\right)\right)$, and the latter statement holds with probability at least $1 - \frac{2\ell}{n}$, so both statements hold simultaneously with probability at least $1 - \frac{21 \log n}{n}$.

		Let $P_1, \ldots, P_r$ be paths satisfying \ref{itm:findpaths-1}--\ref{itm:findpaths-3} of \Cref{claim:findingpaths}, and let $T_0$ be the linear forest whose components are $P_1, \ldots, P_r$.
		We will iteratively construct linear forests $T_1, \ldots, T_{\ell}$ such that the number of leaves, say $t_i$, in $T_i$ decreases quickly as $i$ grows. More precisely, we claim that there is a sequence of linear forests $T_1, \ldots, T_{\ell}$ such that for every $i \in \{1, \ldots, \ell\}$, $T_i$ satisfies the following properties.
  
		\begin{enumerate}[label = \rm(F\arabic*)]
			\item \label{itm:almostspanningpath-1}
				$V(T_i) \subseteq V(T_0) \cup V'_1 \cup \ldots \cup V'_i$, 
			\item \label{itm:almostspanningpath-2}
				The number of leaves $t_i$ in $T_i$ satisfies $2 \le t_{i} \le \max\{2, \frac{9}{10} t_{i-1}\}$, and
			\item \label{itm:almostspanningpath-3}
				the leaves of $T_i$ are contained in $X_1 \cup X_t$.
		\end{enumerate}
		To prove this, we use induction on $i$. For $i = 0$, note that \ref{itm:almostspanningpath-1} and \ref{itm:almostspanningpath-2} are vacuously true, and \ref{itm:almostspanningpath-3} holds due to the choice of the paths $P_1, \ldots, P_r$. Now, let us assume that there is a sequence of linear forests $T_0, \ldots, T_j$ such that $T_j$ satisfies \ref{itm:almostspanningpath-1}--\ref{itm:almostspanningpath-3}, and show that there is a linear forest $T_{j+1}$ satisfying \ref{itm:almostspanningpath-1}--\ref{itm:almostspanningpath-3}.

		If $t_j = 2$, define $T_{j+1} = T_j$; it is then easy to check that $T_{j+1}$ satisfies \ref{itm:almostspanningpath-1}--\ref{itm:almostspanningpath-3}. So we may assume that $t_j > 2$. Then, in fact, $t_j \ge 4$ as the number of leaves in a linear forest is always even. 
		Now, let $Y$ be a set of leaves in $T_j$ obtained by taking exactly one leaf from each path of $T_j$, so that $|Y| = t_j/2 \ge 2$. Note that since $T_j$ satisfies \ref{itm:almostspanningpath-3}, $Y \subseteq X_1 \cup X_t$. 
		Thus, by the assumption that the conclusion of \Cref{claim:connectingverticesofY} holds (with $V_{j+1}', X_1 \cup X_t, q_3, 2q_2$ playing the roles of $V, W, p_1, p_2$, respectively), there is a collection $\mathcal P$ of at least $\frac{t_j/2}{10} = \frac{t_j}{20}$ vertex-disjoint paths of length at most $(\log n)^{c+4}$ whose internal vertices are contained in $V'_{j+1}$, and whose leaves are contained in $Y$. Let $T_{j+1}$ be the linear forest obtained by adding the edges of the paths in $\mathcal P$ to $T_j$. Note that since each path in $\mathcal P$ joins two different paths in $T_j$, it reduces the number of leaves in $T_j$ by exactly $2$, so we have $t_{j+1} = t_j - 2|\mathcal P| \le \max\{2, t_j - 2(t_j/20)\} = \max\{2 ,  9t_j/10\}$, showing that $T_{j+1}$ satisfies \ref{itm:almostspanningpath-2}. Since the set of leaves of $T_{j+1}$ is a subset of the set of leaves in $T_j$ and $T_j$ satisfies \ref{itm:almostspanningpath-3}, it follows that $T_{j+1}$ satisfies \ref{itm:almostspanningpath-3}. Finally, since $V(T_{j+1}) \subseteq V(T_j) \cup V'_{j+1} \subseteq V(T_0) \cup V'_1 \cup \ldots \cup V'_{j+1}$, $T_{j+1}$ satisfies \ref{itm:almostspanningpath-1}. This shows that there is a linear forest $T_{j+1}$ satisfying \ref{itm:almostspanningpath-1}--\ref{itm:almostspanningpath-3}, as desired. 

		Hence, by repeatedly using \ref{itm:almostspanningpath-2} for $1 \le i \le \ell$, we obtain that $2 \le t_{\ell} \le \max\{2, (\frac{9}{10})^{\ell} \, t_0\}$. Since $\ell = 10 \log n$, we have $(\frac{9}{10})^{\ell} \, t_0 \le e^{-\frac{\ell}{10}} \, t_0 \le e^{-\log n} \, n = 1$. So $t_{\ell} = 2$. Therefore, the linear forest $T_{\ell}$ is actually a path. Moreover, $V(T_{\ell}) \subseteq V(T_0) \cup V'_1 \cup \ldots \cup V'_{\ell} \subseteq V' \cup X$ by \ref{itm:almostspanningpath-1}, and the leaves of $T_{\ell}$ are contained in $X_1 \cup X_t$ by \ref{itm:almostspanningpath-3}. Finally, $|X - V(T_{\ell})| \le |X - \bigcup_{i = 1 }^r V(P_i)|\le \frac{1}{2 \log n} |X|$ since $\bigcup_{i = 1 }^r V(P_i) = V(T_0) \subseteq V(T_{\ell})$ and $P_1, \ldots, P_r$ satisfy \ref{itm:findpaths-3}. This shows that $T_{\ell}$ is a path $P$ as required by the claim. 
	\end{proof}

\subsection{Finding an \texorpdfstring{$F$}{F}-subdivision in \texorpdfstring{$H$}{H} containing the nearly Hamilton path \texorpdfstring{$P$}{P}}
\label{sec:findingFsubdivisioncontainingP}
	In this subsection we complete the proof of Lemma~\ref{lem:almostspannningtopological} by finding a subdivision of $F$ that contains the nearly Hamilton path $P$ guaranteed by Claim~\ref{claim:almostspanningpath} with high probability. In fact, we will find a copy of a subdivision of the complete subgraph $K_f$ with $f \coloneqq |V(F)|$ such that a path $P'$ containing the nearly Hamilton path $P$ is one of the $\binom{f}{2}$ paths defining this copy. This is indeed sufficient to prove Lemma~\ref{lem:almostspannningtopological} because $F$ is contained in $K_f$ and we may assume that $P'$ joins two vertices which are adjacent in $F$. To that end, we will need the following well-known result.

	\begin{theorem}[Bollob\'as-Thomason~\cite{bollobas1998proof}, Koml\'os-Szemer\'edi~\cite{komlos1996topological}]
		\label{thm:BT}
		Let $p$ be a positive integer. Then, every graph with average degree at least $512 p^2$ contains a subdivision of the complete graph $K_p$. 
	\end{theorem}

	We are now ready to put everything together to complete the proof of Lemma~\ref{lem:almostspannningtopological}. 

\begin{proof}[ of Lemma~\ref{lem:almostspannningtopological}]
Let $V_0 \coloneqq V' \cup V_1 \cup V_2$ be a random partition of $V_0$ such that each $v \in V_0$ is independently included in $V'$ with probability $1/2$, in $V_1$ with probability $1/4$, and in $V_2$ with probability $1/4$.
		We claim that with probability at least $1 - \frac{22\log n}{n}$ the following five properties hold simultaneously.
		\begin{enumerate}[label = \rm(S\arabic*)]
			\item \label{itm:almostspanningsubdivision-4}
				For every $v \in R$, we have $|N_H(v) \cap R| \ge \frac{q_2 d}{5}$.
			\item \label{itm:almostspanningsubdivision-5}
				$|V_0| \le 1.01 q_1n$.
			\item \label{itm:almostspanningsubdivision-6}
				$|R| \le 1.01 q_2n$.
			\item \label{itm:almostspanningsubdivision-7}
				There is a path $P$ in $H$ such that $V(P) \subseteq V' \cup X$, the leaves of $P$, denoted $u_1, u_2$, are contained in $X_1 \cup X_t$, and $|X \setminus V(P)| \le \frac{1}{2\log n}|X| \le \frac{n}{2\log n}$.
			\item \label{itm:almostspanningsubdivision-8}
    For every pair of distinct vertices $u, v \subseteq X_1 \cup X_t \cup R$, there is a path in $H$ joining the vertices $u$ and $v$ whose internal vertices are in $V_1$, and there is another path in $H$ joining the vertices $u$ and $v$ whose internal vertices are in $V_2$. 
		\end{enumerate}
		Indeed, for \ref{itm:almostspanningsubdivision-4}, recall that every vertex $v \in V(H)$ is included in $R$ with probability $q_2$, and that $\delta(H) \ge \frac{d(H)}{2} \ge \frac{d}{4}$.
		Therefore, for any given vertex $v \in R$, we have $\mathbb{E}[|N_H(v) \cap R|] \ge \frac{q_2 d}{4}$. Thus, \ref{itm:almostspanningsubdivision-4} holds with probability at least $1 - \exp\left(-\Omega\left((\log n)^{2}\right)\right)$, by the Chernoff bound and the union bound.
        Since each vertex $v \in V(H)$ is included in in $V_0$ with probability $q_1$, and in $R$ with probability $q_2$, \ref{itm:almostspanningsubdivision-5} and \ref{itm:almostspanningsubdivision-6} hold with probability at least $1 - \exp\left(-\Omega\left((\log n)^{2}\right)\right)$, by the Chernoff bound. 
		Moreover, by \Cref{claim:almostspanningpath}, \ref{itm:almostspanningsubdivision-7} holds with probability at least $1 - \frac{21\log n}{n}$. Finally, \ref{itm:almostspanningsubdivision-8} holds with probability at least $1 - \frac{4}{n}$, by applying \Cref{claim:connectingverticesofY} twice (with $k = 2$ and $V_i, X_1 \cup X_t \cup R, \frac{q_1}{4}, 3q_2$ playing the roles of $V, W, p_1, p_2$, respectively, for $i \in [2]$) using that $3q_2 \le \frac{q_1}{400}$ and $d \ge \frac{8 (\log n)^{10c+21}}{(3q_2)^{10}}$. This shows that the properties \ref{itm:almostspanningsubdivision-4}--\ref{itm:almostspanningsubdivision-8} hold simultaneously with probability at least at least $1 - \frac{22\log n}{n}$, as desired.

		To prove Lemma~\ref{lem:almostspannningtopological}, we now condition on the properties \ref{itm:almostspanningsubdivision-4}--\ref{itm:almostspanningsubdivision-8} holding, and show how to find a subdivision of $F$ covering all but at most $\frac{n}{\log n}$ vertices of $H$.

		Let $f$ denote the number of vertices of $F$.
		By \ref{itm:almostspanningsubdivision-4}, the average degree of $H[R]$ is at least $\frac{q_2 d}{5} \ge \frac{d}{(\log n)^3} \ge \frac{512 d}{(\log n)^4}$. Therefore, by Theorem~\ref{thm:BT}, $H[R]$ contains a subdivision $K$ of the complete graph of order $\frac{\sqrt{d}}{(\log n)^2} \ge f$. In other words, there are $f$ vertices $v_1, \ldots, v_f \in V(K) \subseteq R$, and $\binom{f}{2}$ paths $P_{i,j}$ in $H[R]$, for $1 \le i < j \le f$, such that $P_{i,j}$ joins $v_i$ and $v_j$, and the interiors of the paths $P_{i,j}$ are pairwise vertex-disjoint and disjoint of $\{v_1, \ldots, v_f\}$. 

		Recall that, by \ref{itm:almostspanningsubdivision-7}, there is a path $P$ in $H$ joining vertices $u_1, u_2 \in X_1 \cup X_t$ such that $V(P) \subseteq V' \cup X$ and $|X \setminus V(P)| \le \frac{n}{2\log n}$. Our plan is to replace the path $P_{1,2}$ joining $v_1$ and $v_2$ in the subdivision $K$ with a path that contains $P$ as a subpath. By \ref{itm:almostspanningsubdivision-8}, we know that there is a path $P'_{u_1v_1}$ in $H$ joining the vertices $u_1$ and $v_1$ whose internal vertices are in $V_1$, and there is a path $P'_{u_2v_2}$ in $H$ joining the vertices $u_2$ and $v_2$ whose internal vertices are in $V_2$. Now let $P'_{1,2} \coloneqq P \cup P'_{u_1v_1} \cup P'_{u_2v_2}$. Observe that $P'_{1,2}$ is a path in $H$ joining $v_1$ and $v_2$, and $V(P'_{1,2}) \cap R = \{ v_1, v_2 \}$ since $V(P) \subseteq V' \cup X$ by \ref{itm:almostspanningsubdivision-7}. Hence, $(V(P'_{1,2}) \setminus \{ v_1, v_2 \}) \cap V(P_{i,j}) = \emptyset$ for every $1 \le i < j \le f$. 

		Let $K'$ be the subgraph of $H$ obtained by replacing the path $P_{1,2}$ in the subdivision $K$ with the path $P'_{1,2}$ (see Figure~\ref{fig:almostspanningsubdivision}). 
		Then, $K'$ is also a subdivision of the complete graph of order $\frac{\sqrt{d}}{(\log n)^2} \ge f$. 
		It is easy to see that (using that $F$ is non-empty) by omitting some of the paths $P_{i,j}$ if necessary, but keeping the path $P_{1,2}'$, we can obtain a subdivision $F'$ of $F$, we have
		\begin{align} \label{eqn:Kprime}
			\begin{split}
				|V(H) \setminus V(F')| 
				\le |V_0|+|X \setminus V(P)|+ |R| 
				&\le 1.01 q_1 n + \frac{n}{2\log n} + 1.01 q_2 n \\ 
				&= \frac{1.01n}{3\log n} + \frac{n}{2 \log n} + \frac{6 \cdot 1.01 n}{(\log n)^{3}} \le \frac{n}{\log n},
			\end{split}
		\end{align}
		where we used that the omitted paths are contained in $R$ and \ref{itm:almostspanningsubdivision-5}--\ref{itm:almostspanningsubdivision-7} for the second inequality. 
		This completes the proof of Lemma~\ref{lem:almostspannningtopological}.         
\end{proof}

\section{Applications}
\label{sec:application:chords}

In this section, we apply the ideas developed in the previous sections to prove our main result and present two further applications of our methods to other problems, each described in a separate short subsection. The latter two applications rely on the fact that every sufficiently regular expander with large enough average degree contains a nearly Hamilton cycle—a simple consequence of our methods (obtained by letting $F$ be a triangle in \Cref{lem:almostspannningtopological}).

\subsection{Packing a regular graph with \texorpdfstring{$F$}{F}-subdivisions}
\label{sec:proofofmainresult}

	In this subsection we prove our main result (Theorem~\ref{thm:packingsubdivisons}) by combining \Cref{cor:expandercoverstrong} and Lemma~\ref{lem:almostspannningtopological}. The former lemma shows that one can cover almost all vertices of any regular graph with sufficiently regular expanders, and the latter lemma guarantees an almost-spanning $F$-subdivision  within each of these expanders.

	\begin{proof}[ of Theorem~\ref{thm:packingsubdivisons}]
        Let $G$ be a $d$-regular graph of order $n$ with $d \ge (\log n)^{130}$. First, we plan to apply \Cref{cor:expandercoverstrong} to $G$. To that end, let $\alpha = \frac{1}{28}$, $C = 6$, so that $c = \frac{C(C-1)}{C - 28 \alpha-1} = 7.5$ and $C - 28 \alpha-1 = 4$. Let $\eps = (\log n)^{-5}$ so that $0 < \eps = (\log n)^{-(C-1)}$. Since $G$ is $d$-regular, we have $d(G) = d \ge d(1 - \eps)$ and $\Delta(G) \le d$. Moreover, note that $\frac{C - 28 \alpha-1}{C-1} = \frac{4}{5} \ge \frac{1}{2}$, and $C \ge 28 \alpha + 3 = 4$. Hence, by \Cref{cor:expandercoverstrong}, there is a collection $\cH$ of vertex-disjoint subgraphs of $G$ such that every $H \in \cH$ is a $(\frac{1}{8}, c, s_H)$-expander satisfying $d(H) \ge d(1 - \eps_H)$ and $\delta(H) \ge \frac{d(H)}{2}$, where $s_H = \frac{d}{4 (\log |V(H)|)^{c}}$, and $\eps_H = (\log |V(H)|)^{-4}$. Moreover, 
		\begin{equation}
			\label{eq:coverexp}
			\sum_{H \in \cH} |V(H)| \ge \left(1 - \frac{(\log \log \log n)^2}{(\log \log n)^{1/28}}\right)n.
		\end{equation}

		Let $H \in \cH$ be an arbitrary member of the collection $\cH$. Since $G$ is $d$-regular and $H$ is a subgraph of $G$, we have $\Delta(H) \le d$. Hence, by applying Lemma~\ref{lem:almostspannningtopological} to $H$ with $\eps_H$ and $s_H$ playing the roles of $\eps$ and $s$, respectively, we obtain a copy $K_H$ of an $F$-subdivision in $H$ such that 
		\begin{align}
			\label{eq:covertopolo}
			\begin{split}
				|V(K_H)| 
				\ge \left(1 - \frac{1}{\log |V(H)|}\right) |V(H)| 
				& \ge \left(1 - \frac{1}{\log d(H)}\right) |V(H)|  \\
				& \ge \left(1 - \frac{1}{\log \log n}\right) |V(H)|.
			\end{split}
		\end{align}
		Note that Lemma~\ref{lem:almostspannningtopological} is indeed applicable because $d \ge (\log n)^{130} \ge (\log |V(H)|)^{10 c + 51}$. Now consider the collection $\{ K_H \mid H \in \cH \}$ of vertex-disjoint copies of $F$-subdivisions in $G$. By \eqref{eq:coverexp} and \eqref{eq:covertopolo}, we have
		\begin{equation*}
			\sum_{H \in \cH} |V(K_H)| \ge \left(1 - \frac{1}{\log \log n}\right)\left(1 - \frac{(\log \log \log n)^2}{(\log \log n)^{1/28}}\right)n \ge \left(1 - \frac{1}{(\log \log n)^{1/30}}\right)n.
		\end{equation*}
		Hence $\{ K_H \mid H \in \cH \}$ is the desired $\TF$-packing in $G$, covering all but at most $\frac{n}{(\log \log n)^{1/30}}$ vertices of $G$. This completes the proof of Theorem~\ref{thm:packingsubdivisons}.
	\end{proof}

	\begin{rem}
		Note that the proof of Theorem~\ref{thm:packingsubdivisons} actually shows that it suffices for a graph to be nearly regular (rather than regular) to find a $\TF$-packing covering almost all of its vertices. More precisely, it shows that 
		any graph $G$ such that $d(G) \ge d(1 -\frac{1}{(\log n)^{5}})$, $\Delta(G) \le d$ and $d \ge (\log n)^{130}$ contains a $\TF$-packing which covers all but at most $\frac{n}{(\log \log n)^{1/30}}$ vertices of $G$. 
	\end{rem}

\subsection{Cycle partitions of regular graphs}
\label{subsec:cyclepartitions}

Our next application concerns the well-known conjecture of Magnant and Martin~\cite{magnant2009note}, asserting that every $d$-regular graph contains a collection of at most $\frac{n}{d+1}$ pairwise vertex-disjoint paths that cover all vertices. We establish an asymptotic version of this conjecture in a stronger form for sufficiently large $d$, allowing for $o(n)$ uncovered vertices but covering the vertices with cycles instead of paths.

\begin{theorem} \label{thm:cycle-partition}
	Let $G$ be a $d$-regular graph of order $n$, where $n$ is large enough and $d \ge (\log n)^{130}$.
	Then, there is a collection of at most $\frac{n}{d+1}$ vertex-disjoint cycles covering all but at most $O\left(\frac{n}{(\log \log n)^{1/30}}\right)$ vertices of $G$.
\end{theorem}

Recently, Montgomery, M\"uyesser, Pokrovskiy, and Sudakov~\cite{montgomery2024approximate} proved a result closely related to \Cref{thm:cycle-partition}. 
Their theorem is stronger in that it does not impose any lower bound on $d$, 
but it yields a slightly weaker conclusion, producing a collection of paths rather than cycles. 
Moreover, their approach does not rely on sublinear expanders and is therefore quite different from the methods used in our work.

\begin{proof}[ of \Cref{thm:cycle-partition}]
	Write $C \coloneqq 6$ and $c \coloneqq \frac{C(C-1)}{C-2} = 7.5$. We apply \Cref{cor:expandercoverstrong} to the graph $G$ with parameters $\alpha = \frac{1}{28}$, $\eps = 0$, $n$, $d$, $C$ and $c$. Let $\cH$ be the resulting collection of expanders guaranteed by the lemma, so that every $H \in \cH$ is a $(\frac{1}{8}, c, s_H)$-expander satisfying $d(H) \ge (1 - \eps_H) d$ and $\delta(H) \ge d(H)/2$, where $s_H \coloneqq \frac{d}{4 (\log |V(H)|)^{c}}$ and $\eps_H \coloneqq (\log |V(H)|)^{-(C - 2)} = (\log |V(H)|)^{-4}$, and, moreover, $\sum_{H \in \cH}|V(H)| \ge (1 - \frac{(\log \log \log n)^2}{(\log \log n)^{1/28}})n$. Note that $d \ge (\log n)^{130} \ge (\log |V(H)|)^{10c+51}$.

	Next, we apply \Cref{lem:almostspannningtopological} to each $H \in \cH$, with $c$, $\eps_H$, $d$, $s_H$, and $|V(H)|$ taking the roles of $c$, $\eps$, $d$, $s$, and $n$, respectively (noting that these parameters satisfy the lemma's requirements), and with $F$ being a triangle. For each $H \in \cH$, this yields a cycle $C_H$ in $H$ that covers all but at most $\frac{|V(H)|}{\log |V(H)|}$ vertices of $H$, meaning that $|C_H| \ge (1 - \frac{1}{\log |V(H)|})|V(H)|$. Since $|V(H)| \ge d(H) \ge (1 - \eps_H)d = (1 - \frac{1}{(\log |V(H)|)^4})d$, this implies that $|C_H| \ge (1 - \frac{1}{\log |V(H)|} - \frac{1}{(\log |V(H)|)^4})d \ge (1 - \frac{1}{\log \log n})d$ (where in the last inequality we used that $\log |V(H)| \ge \log (d/2) \ge 2\log \log n$). Then, $\cC \coloneqq \{C_H : H \in \cH\}$ is a collection of vertex-disjoint cycles of length at least $(1 - \frac{1}{\log \log n})d$, covering all but at most $O\left(\frac{n}{(\log \log n)^{1/30}}\right)$ vertices of $G$ as shown by the following.
	\begin{align} \label{eqn:small-remainder}
		\begin{split}
		\sum_{H \in \cH}\frac{|V(H)|}{\log |V(H)|} + \left(n - \sum_{H \in \cH} |V(H)|\right)
		& \le \sum_{H \in \cH}\frac{|V(H)|}{\log \log n} + \frac{n (\log \log \log n)^2}{(\log \log n)^{1/28}} \\
		& \le n \left(\frac{1}{\log \log n} + \frac{(\log \log \log n)^2}{(\log \log n)^{1/28}}\right)
		= O\left(\frac{n}{(\log \log n)^{1/30}}\right).
		\end{split}
	\end{align}
	Let $\cC' \subseteq \cC$ be a subcollection of cycles consisting of $\min\{\frac{n}{d+1}, |\cC|\}$ cycles in $\cC$. We claim that $\cC'$ covers all but $O(\frac{n}{(\log \log n)^{1/30}})$ vertices in $G$. Indeed, this is clear from the choice of $\cC$ and \eqref{eqn:small-remainder} if $|\cC'| = |\cC|$ (i.e.\ if $\cC' = \cC$). Otherwise, $\cC'$ is a collection of $\frac{n}{d+1}$ vertex-disjoint cycles of length at least $(1 - \frac{1}{\log \log n})d$, implying that the cycles in $\cC'$ cover at least $(1 - \frac{1}{\log \log n}) d \cdot \frac{n}{d+1} \ge (1 - \frac{2}{\log \log n}) n = n - O(\frac{n}{\log \log n})$ vertices of $G$, as desired.
\end{proof}

\subsection{Cycles with many chords}
\label{subsec:cyclewithmanychords}

	Our last application is about finding a cycle with many chords, addressing an old question of Chen, Erd\H{o}s and Staton~\cite{chen1996proof}. 

\begin{corollary}
\label{cor:chords}
If $n$ is sufficiently large, then every $n$-vertex graph with at least $n(\log n)^{130}$ edges contains a cycle $C$ with at least $|C|$ chords.
\end{corollary}

This provides an alternative proof of a recent result of Dragani\'c, Methuku, Munh\'a Correia, 
and Sudakov~\cite{draganic2023cycles}, who obtained a stronger bound with a smaller 
polylogarithmic factor. Specifically, they showed that $\Omega(n(\log n)^8)$ edges suffice to 
guarantee the existence of a cycle $C$ with at least $|C|$ chords. 
While their argument also relies on sublinear expansion, it is based on self-avoiding 
random walks and is therefore quite different from the approach used here.

We now prove \Cref{cor:chords} by combining \Cref{cor:findexpander}, which shows the existence of a nearly regular expander in a graph with sufficiently large average degree, with \Cref{lem:almostspannningtopological}, which finds a nearly Hamilton cycle within such an expander.

\begin{proof}[ of Corollary~\ref{cor:chords}]
	Let $G$ be an $n$-vertex graph with at least $n(\log n)^{130}$ edges. 
	Applying \Cref{cor:findexpander}, with $C = 6$ and $d = (\log n)^{126}$, we obtain a subgraph $H \subseteq G$ which is a $(\frac{1}{8}, c, s)$-expander satisfying $\Delta(H) \le d$, $d(H) \ge (1 - \mu)d$ and $\delta(H) \ge d(H)/2$, where $c = 7.5$, $s \coloneqq \frac{d}{4 (\log m)^c}$, $\mu = (\log m)^{-4}$,  and $m \coloneqq |V(H)|$.

	Applying Lemma~\ref{lem:almostspannningtopological} to $H$ with $K_3$ playing the role of $F$, we obtain a cycle $C$ in $H$ such that 
	\begin{equation*}
		|V(H) \setminus V(C)| \le \frac{m}{\log m}.
	\end{equation*}
	(Note that Lemma~\ref{lem:almostspannningtopological} is indeed applicable here because $d = (\log n)^{126} = (\log n)^{10c + 51} \ge (\log m)^{10 c + 51}$.) Hence, at most $\frac{dm}{\log m}$ edges of $H$ are incident to vertices of $H$ not contained in $C$. This implies that the number of edges spanned by $C$ is at least
	\begin{equation*}
		\frac{d(1-\mu)m}{2} - \frac{dm}{\log m} 
		\ge  \frac{dm}{4} \ge  2|C|.
	\end{equation*}
	Here, we rely on the fact that $m$ is large, which follows from $m \ge d(H) \ge (1 - \mu)d = \Omega((\log n)^{126})$ and the assumption that $n$ is sufficiently large.
	This shows that $C$ is a cycle with at least $|C|$ chords, completing the proof of Corollary~\ref{cor:chords}.
\end{proof}

\section{Concluding remarks}

One very useful feature of \Cref{cor:expandercoverstrong} is that it allows us to extract a robust sublinear expander from a regular graph while retaining quantitative control over its regularity, albeit at the expense of weakening its expansion. More precisely, it shows that for any $c_1 \ge 2$ there exists a constant $c_2>0$ such that every $d$-regular graph $G$ of sufficiently large degree (at least polylogarithmic in the number of vertices) contains a sublinear expander $H$ with average degree at least $d\bigl(1-\frac{1}{(\log |V(H)|)^{c_1}}\bigr)$ and expansion factor $\frac{1}{(\log |V(H)|)^{c_2}}$. Thus, by choosing $c_1$ large, we may enforce strong regularity properties for $H$. However, this comes at the cost of increasing $c_2$, and hence weakening the expansion of the resulting expander. From the perspective of potential applications, it would be particularly interesting to determine whether one can obtain significantly stronger expansion—say, of order $\frac{1}{(\log |V(H)|)^{2}}$—while still maintaining good control over the regularity of the expander.

Moreover, the notion of regularity employed throughout this paper requires that the average degree of $H$ is close to its maximum degree. While this condition suffices for our purposes, a more canonical notion of almost-regularity would demand that the minimum degree is also close to the maximum degree, that is, $\Delta(H)/\delta(H)=1+o(1)$. It would therefore be very interesting to establish whether every $d$-regular graph contains an almost regular expander in this stronger sense. In particular, one may ask whether every $d$-regular graph necessarily contains a \emph{regular} robust sublinear expander.

A further natural question is whether one can obtain a perfect cover using expanders. Specifically, does every $n$-vertex $d$-regular graph admit a decomposition into vertex-disjoint robust sublinear expanders? While \Cref{cor:expandercoverstrong} yields a cover of all but $o(n)$ vertices, it remains open whether this error term can be eliminated. It would also be very interesting to determine whether the assumption $d \ge 2\log n$ in \Cref{cor:expandercoverstrong} can be removed, or at least substantially weakened.



\textbf{Acknowledgements.} We thank Dong-yeap Kang for bringing~\cite{kuhn2005packings} to our attention. 
We are grateful to Matija Buci\'c, Nemanja Dragani\'c, Oliver Janzer, Dong-yeap Kang, Richard Montgomery, 
Alp M\"uyesser, and Jacques Verstra\"ete for many helpful discussions related to this work. 
We are also grateful to the anonymous referee for a careful reading of our paper and for providing 
detailed comments that significantly improved the presentation.

\bibliographystyle{abbrv}
\bibliography{references}

\appendix

\end{document}